\documentclass{article}
\usepackage{graphicx}
\usepackage[left=3cm, top=2cm, right=3cm, bottom=3cm]{geometry}
\usepackage{listings}
\usepackage{float}
\usepackage{amssymb,amsmath,mathtools}
\usepackage{enumerate}
\usepackage[numbers,square,sort]{natbib}
\usepackage{wrapfig}
\usepackage{framed}
\usepackage{tabstackengine}
\usepackage{braket}
\usepackage[nottoc]{tocbibind}\settocbibname{References}
\usepackage{mathrsfs}

\usepackage{indentfirst}
\usepackage{setspace}
\usepackage{bm}

\usepackage{hyperref}
\hypersetup{colorlinks,citecolor=black,linkcolor=black}

\usepackage[dvipsnames]{xcolor}

\usepackage{tikz}
\tikzset{every loop/.style={}}
\usetikzlibrary{arrows,shapes,decorations.markings,automata,backgrounds,petri,bending,calc,angles}
\usepackage{circuitikz}

\usepackage{fancyhdr}
\fancypagestyle{firststyle}
{
	
	\fancyfoot[l]{\footnotesize\textit{Date:} \today}
	\fancyfoot[c]{1}
	\fancyhead{}

}

\usepackage{tikz-cd} 
\tikzset{
    labl/.style={anchor=south, rotate=90, inner sep=.5mm}
}
\tikzcdset{row sep/normal=2.7em}
\tikzcdset{column sep/normal=3.6em}

\usepackage{enumitem}
\setlist[enumerate,1]{label=(\arabic*)}
\setlist[enumerate,2]{label={(\alph*)},ref={(\alph*)}}
\setlist[enumerate,3]{label={(\roman*)},ref={(\roman*)}}
\newlist{steplist}{enumerate}{1}
\setlist[steplist]{label={Step \arabic*:}, ref={Step \arabic*}}

\makeatletter
\newcommand*{\wackyenum}[1]{%
	\expandafter\@wackyenum\csname c@#1\endcsname%
}
\newcommand*{\@wackyenum}[1]{%
	$\ifcase#1\or(1')\or(2')\or(3a')\or(3b')\or(4')%
	\else\@ctrerr\fi$%
}
\AddEnumerateCounter{\wackyenum}{\@wackyenum}{(1')}
\makeatother

\usepackage[utf8]{inputenc}
\usepackage[english]{babel} 

\usepackage{amsthm}
\newtheorem{thm}{Theorem}[section]
\newtheorem{lem}[thm]{Lemma}
\newtheorem{prop}[thm]{Proposition}
\newtheorem{cor}[thm]{Corollary}
\numberwithin{equation}{section}

\theoremstyle{definition}
\newtheorem{defn}[thm]{Definition} 
\newtheorem{remk}[thm]{Remark}
\newtheorem{nota}[thm]{Notation}
\newtheorem{cons}[thm]{Construction}

\newcommand{\cA}{\mathcal{A}}

\newcommand{\cP}{\mathcal{P}}

\newcommand{\fR}{\mathfrak{R}}

\newcommand{\La}{\Lambda}

\newcommand{\hH}{\hat{\mathcal{H}}}

\newcommand{\acts}{\curvearrowright}

\newcommand{\Hull}{\operatorname{Hull}}

\newcommand{\link}{\operatorname{link}}

\newcommand{\wt}[1]{\widetilde{#1}}
\newcommand{\ti}[1]{\tilde{#1}}

\makeatletter
\newsavebox{\@brx}
\newcommand{\llangle}[1][]{\savebox{\@brx}{\(\m@th{#1\langle}\)}%
	\mathopen{\copy\@brx\kern-0.5\wd\@brx\usebox{\@brx}}}
\newcommand{\rrangle}[1][]{\savebox{\@brx}{\(\m@th{#1\rangle}\)}%
	\mathclose{\copy\@brx\kern-0.5\wd\@brx\usebox{\@brx}}}
\makeatother

\begin{document}
\begin{center}
{\LARGE\bf
Product separability for special cube complexes}\\
\bigskip
\bigskip
{\large Sam Shepherd\footnote{The author was supported by the CRC 1442 Geometry: Deformations and Rigidity, and by the Mathematics Münster Cluster of Excellence}}
\end{center}

\begin{abstract}
	We prove that if a group $G$ admits a virtually special action on a CAT(0) cube complex, then any product of convex-cocompact subgroups of $G$ is separable. Previously, this was only known for products of three subgroups, or in the case where $G$ is hyperbolic, or in some other more technical cases with additional assumptions on the subgroups (plus these previous results assume that the action of $G$ is cocompact).
	We also provide an application to the action of a virtually special cubulated group on its contact graph (and to some other actions of cubulated groups on graphs).
\end{abstract}
\tableofcontents

\bigskip
\section{Introduction}

The main theorem of this paper is as follows.

\begin{thm}\label{thm:prodsep}
	Let $G\acts \wt{X}$ be a virtually special action of a group $G$ on a CAT(0) cube complex $\wt{X}$, and let $K_1,K_2,\dots,K_n<G$ be convex-cocompact subgroups. Then the product $K_1K_2\cdots K_n$ is separable in $G$.
\end{thm}

Here, an action $G\acts \wt{X}$ on a CAT(0) cube complex is \emph{virtually special} if there exists a finite-index subgroup $G'<G$ acting freely on $\wt{X}$, with the quotient $\wt{X}/G'$ a special cube complex (in this paper actions on cube complexes will always be by cubical automorphisms).
A \emph{convex-cocompact subgroup} of $G$ is a subgroup that stabilizes and acts cocompactly on a convex subcomplex of $\wt{X}$.
A subset $S$ of a group $G$ is \emph{separable} if, for any $g\in G- S$, there is a homomorphism $\phi:G\to F$ to a finite group $F$ with $\phi(g)\notin\phi(S)$.

If $G\acts \wt{X}$ is a proper and cocompact action on a CAT(0) cube complex, then we call  $G\acts \wt{X}$ a \emph{cubulated group}. We note that in this case it is often easier to construct convex-cocompact subgroups, for instance any hyperplane stabilizer will be convex-cocompact.

The separability of convex-cocompact subgroups of a virtually special cubulated group $G\acts \wt{X}$ is one of the cornerstones of the theory of special cube complexes, and was proved by Haglund and Wise in their first paper on special cube complexes \cite[Corollary 7.9]{HaglundWise08}.
Theorem \ref{thm:prodsep} is a natural generalisation of this result.
If $G$ is hyperbolic and cubulated, then the cubulation is virtually special by \cite{Agol13}, so in this case one deduces that every quasiconvex subgroup of $G$ is separable (a subgroup is quasiconvex if and only if it is convex-cocompact with respect to the cubulation by \cite[Proposition 7.2]{HaglundWise08}).
An astonishing array of groups have virtually special cubulations, including finite volume hyperbolic 3-manifold groups, hyperbolic free-by-cyclic groups and many small cancellation groups \cite{Agol13,WiseQCH,HagenWise15,Wise04}, thus providing many examples of groups with strong separability properties.
In fact, for finite-volume hyperbolic 3-manifold groups, the separability of quasiconvex subgroups provided the critical final step in the resolution of the Virtual Haken and Virtual Fibering conjectures \cite[\S9]{Agol13}.

The separability of a product of two convex-cocompact subgroups has also played an important role in the theory.
Haglund and Wise proved that a cubulated group $G\acts \wt{X}$ is virtually special if and only if the hyperplane subgroups and products of two hyperplane subgroups (corresponding to pairs of intersecting hyperplanes) are separable in $G$ \cite[Theorem 9.19]{HaglundWise08}.
Furthermore, Reyes proved that if $G\acts \wt{X}$ is a virtually special cubulated group then the product of any two convex-cocompact subgroups is separable in $G$ \cite[Theorem A.1]{OregonReyes23}. This was one of the ingredients in Reyes' theorem about virtual specialness of relatively hyperbolic cubulated groups \cite[Theorem 1.2]{OregonReyes23}.
Going another step further, separability of products of three convex-cocompact subgroups was proved by the author in \cite[Proposition 4.13]{Shepherd23}, and this was a key tool in the proof of \cite[Theorem 1.2]{Shepherd23} about commanding convex elements.

The separability of products of arbitrarily many subgroups has also been studied in certain contexts with ``hyperbolic-like'' behaviour. Ribes--Zalesskii proved separability for any product of finitely generated subgroups of a finite-rank free group \cite{RibesZalesskii93}.
Minasyan proved separability for any product of quasiconvex subgroups in a QCERF hyperbolic group \cite{Minasyan06}, which applies to any hyperbolic cubulated group.
This was generalized to the relatively hyperbolic setting by McClellan \cite{McClellan19} and Minasyan--Mineh \cite{MinasyanMineh25}.
In particular, \cite[Corollary 4.5]{McClellan19} combined with \cite[Theorem 1.1]{SageevWise15} implies separability for any product of full relatively quasiconvex subgroups in a relatively hyperbolic virtually special cubulated group.
One can also drop the fullness assumption in some situations, for example combining \cite[Theorem 1.8]{MinasyanMineh25} with \cite[Corollary 6.4]{SageevWise15} implies separability for any product of relatively quasiconvex subgroups in a cubulated group which is hyperbolic relative to virtually abelian subgroups (virtual specialness follows from \cite[Corollary 1.3]{OregonReyes23} in this case).
In a slightly different direction, Mineh--Spriano showed separability for any product of stable subgroups in a virtually special cubulated group \cite{MinehSpriano24} -- in this case the subgroups exhibit ``hyperbolic-like'' behaviour but the ambient group might not.
Combining Theorem \ref{thm:prodsep} with \cite[Proposition 4.2]{Genevois19}, we obtain the following generalization of Mineh--Spriano's result (in fact this also generalizes the result about full relatively quasiconvex subgroups mentioned above).

\begin{cor}\label{cor:relhyp}
	Let $G\acts \wt{X}$ be a virtually special cubulated group, and let $K_1,K_2,\dots,K_n<G$ be Morse subgroups. Then the product $K_1K_2\cdots K_n$ is separable in $G$.
\end{cor}

Another important class of examples for Theorem \ref{thm:prodsep} is the class of (finitely generated) right-angled Artin and right-angled Coxeter groups. These are virtually special cubulated groups with respect to their standard actions on CAT(0) cube complexes.
Moreover, the previous product separability results from the literature discussed above only apply to some right-angled Artin and Coxeter groups and/or some subgroups of them.
Additionally, Niblo--Reeves showed that any Coxeter group $G$ (not necessarily right-angled) admits a proper action on a locally finite CAT(0) cube complex \cite{NibloReeves03}, and Haglund--Wise showed that this action is virtually special \cite{HaglundWise10}, so Theorem \ref{thm:prodsep} applies here as well.

\subsection{Separability with respect to group actions}

Separability properties are often applied via group actions -- indeed this is the case for many of the applications discussed earlier in the introduction.
One can also formulate a general definition of separability in the context of group actions as follows.
Given an action of a group $G$ on a set $A$, we say that a subset $B\subset A$ is \emph{$G$-separable} if for any $a\in A-B$ there exists a finite-index subgroup $\hat{G}<G$ such that $\hat{G}\cdot a \cap B=\emptyset$.

Naturally, there is a close connection between separability of subsets of $G$ and $G$-separability of subsets of $A$.
For instance, if we consider the action of $G$ on itself by left multiplication, then a subset of $G$ is separable if and only if it is $G$-separable (see the version of separability in Lemma \ref{lem:sepequiv}\ref{it:sep3}).
More generally, if $G$ acts on a set $A$ with finite points stabilisers, then $G$ is residually finite if and only if every finite subset of $A$ is $G$-separable.
Subgroup separability also plays a natural role: for an arbitrary action of $G$ on a set $A$, a singleton subset $\{a\}\subset A$ is $G$-separable if and only if the stabiliser $G_a$ is separable.
Furthermore, in the case where $A$ is a graph with cocompact $G$-action, we have a relation between product separability in $G$ and $G$-separability of balls in $A$, as we demonstrate with the following proposition.

\begin{prop}\label{prop:ballsep}
	Let $G$ act cocompactly on a graph $X$. Suppose that any product of vertex stabilisers is separable in $G$.
	Then for any $x\in VX$ and any integer $R\geq0$, the ball
	$$B_R(x)=\{y\in VX\mid d(x,y)\leq R\}$$
	is $G$-separable.
\end{prop}

We then obtain the following as a corollary of Theorem \ref{thm:prodsep} and Proposition \ref{prop:ballsep}.
See Section \ref{sec:actions} for further explanation.

\begin{cor}\label{cor:contact}
	Let $G\acts X$ be one of the following actions:
	\begin{enumerate}
		\item\label{it:contact} The action of a virtually special cubulated group on its contact graph \cite{Hagen14,BehrstockHagenSisto17}.
		\item\label{it:extension} The action of a finitely generated right-angled Artin group on its extension graph $\Gamma^e$ or clique graph $\Gamma^e_k$ \cite{KimKoberda13}.
		\item\label{it:JSJ} The action of a one-ended hyperbolic cubulated group on its canonical JSJ tree \cite{Bowditch98}.
		\item\label{it:coned} The action of a virtually special cubulated group on its coned-off Cayley graph with respect to a given convex-cocompact subgroup and finite generating set \cite{Farb98} (we make no assumption of relative hyperbolicity here).
	\end{enumerate}
Then for any $x\in VX$ and any integer $R\geq0$, the ball $B_R(x)$ is $G$-separable.
\end{cor}

\subsection{Ideas in the proof}

We begin by remarking that most of the previous results on product separability exploit some form of ``hyperbolic-like'' behaviour. Theorem \ref{thm:prodsep} has no such behaviour present, so the methods are necessarily quite different. Instead, the proof builds on the cubical methods developed in the author's earlier work \cite{Shepherd23}, where separability for products of three convex-cocompact subgroups was proved.
However, going from three to four (and more) subgroups turned out to be quite challenging, and required a number of additional ideas and techniques.

The first step is to reformulate Theorem \ref{thm:prodsep} as a more geometric statement. We do this with the notion of a \emph{route}. Given a non-positively curved cube complex $X$, a \emph{route in $X$ of length $n$} is a sequence of $n$ locally convex subcomplexes in $X$ such that consecutive subcomplexes overlap (this is a slightly simplified version of Definition \ref{defn:routes}, but will suffice for the purposes of the introduction). We choose basepoints in each of these overlaps, plus additional basepoints in the first and last subcomplexes -- which we call the \emph{initial vertex} and \emph{terminal vertex} respectively. A route is \emph{closed} if the initial and terminal vertices are the same, and it is \emph{finite} if the subcomplexes are finite.
A closed route $\fR$ is called \emph{essential} if every path along $\fR$ (i.e. a path with one segment in each of the subcomplexes of $\fR$) is essential.
One can define \emph{elevations} of routes to any cover $\hat{X}\to X$, in analogy with elevations of subcomplexes. 
We can then show that Theorem \ref{thm:prodsep} reduces to the following theorem about elevations of routes.

\theoremstyle{plain}
\newtheorem*{thm:routesext}{Theorem \ref{thm:routesext}}
\begin{thm:routesext}
Let $X$ be a virtually special cube complex.
For any finite essential closed route $\fR$ in $X$, there exists a finite-sheeted cover $\hat{X}\to X$ with no closed elevations of $\fR$. 
\end{thm:routesext}

One of the main tools in the proof of Theorem \ref{thm:routesext} is the Walker and Imitator construction, which was introduced in \cite{Shepherd23}. The construction is essentially equivalent to the canonical completion and retraction of Haglund--Wise \cite{HaglundWise08}, but is more helpful for facilitating certain geometric arguments.
The basic idea is as follows: we have a locally convex subcomplex $Y$ of a finite directly special cube complex $X$ (direct specialness is a slight strengthening of specialness, see Definition \ref{defn:special}), and the walker and imitator wander around the 1-skeleta of $X$ and $Y$ respectively. However, at each step the imitator tries to copy the walker by traversing an edge parallel to the edge traversed by the walker (if not possible, the imitator stays still). It turns out that one can use this setup to construct a finite cover of $X$ with certain properties (see Section \ref{sec:walker} for more details).
In the context of Theorem \ref{thm:routesext} (after reducing to the case where $X$ is finite and directly special), we take $Y$ to be the first subcomplex of the route $\fR$.
We then show that the cover produced by the Walker--Imitator construction satisfies the conclusion of Theorem \ref{thm:routesext} provided that the walker--imitator movements satisfy the following property: if the walker and imitator start together at the initial vertex of $\fR$ and the walker traverses a path along $\fR$, then the imitator does not end up back at the initial vertex.

The walker--imitator property above is satisfied if the route $\fR$ and the hyperplanes of $X$ have certain geometric properties (see Lemma \ref{lem:routetrap}). For routes of length three it is not too hard to reduce to the case where these properties hold, but it becomes much harder for routes of length four and above. This is the source of difficulty for going from three to four subgroups in Theorem \ref{thm:prodsep}, as mentioned earlier.
A large part of this paper is thus devoted to reducing Theorem \ref{thm:routesext} to the case where $\fR$ does satisfy the required properties.
This reduction involves a careful inductive argument, inducting on the length of $\fR$ as well as on several other quantities (this argument takes place across Section \ref{sec:control}).

In one crucial part of the above reduction argument, we take an elevation $\wt{\fR}$ of $\fR$ to the universal cover $\wt{X}$ of $X$, modify $\wt{\fR}$, and then descend it back down to produce a new route in $X$.
The modification of $\wt{\fR}$ involves projecting certain subcomplexes of $\wt{\fR}$ to a certain subcomplex of $\wt{X}$. Here we exploit the theory of projection maps and product subcomplexes in CAT(0) cube complexes.
However, in addition to the existing theory, we require the following two propositions, which we believe to be of independent interest.
In what follows, $\Pi_{\wt{A}}$ denotes the projection map to a subcomplex $\wt{A}$ of a CAT(0) cube complex. See Sections \ref{sec:prelims} and \ref{sec:projections} for the relevant definitions.

\theoremstyle{plain}
\newtheorem*{prop:projection}{Proposition \ref{prop:projection}}
\begin{prop:projection}
	Let $A$ and $B$ be locally convex subcomplexes in a weakly special cube complex $X$, and let $G=\pi_1(X)$. Let $\wt{X}\to X$ be the universal cover and let $\wt{A},\wt{B}\subset\wt{X}$ be elevations of $A$ and $B$. Let $\wt{C}=\Pi_{\wt{A}}(\wt{B})$.
	Then $G_{\wt{C}}=G_{\wt{A}}\cap G_{\wt{B}}$ and $\wt{C}/G_{\wt{C}}$ embeds in $X$.	 
\end{prop:projection}

The second proposition concerns orthogonal complements in CAT(0) cube complexes, a notion originally due to Hagen--Susse \cite{HagenSusse20}.
In a CAT(0) cube complex $\wt{X}$, the \emph{orthogonal complement} of a convex subcomplex $\wt{A}$ at a vertex $\tilde{x}\in\wt{A}$ is the unique largest convex subcomplex $\wt{B}$ containing $\tilde{x}$ such that there is an isometric embedding $\phi:\wt{A}\times\wt{B}\xhookrightarrow{}\wt{X}$ (which identifies $\wt{A}\times\{\tilde{x}\}$ with $\wt{A}$ and $\{\tilde{x}\}\times\wt{B}$ with $\wt{B}$).
See Section \ref{sec:projections} for more background, including other equivalent definitions of orthogonal complement.

\theoremstyle{plain}
\newtheorem*{prop:orthcomp}{Proposition \ref{prop:orthcomp}}
\begin{prop:orthcomp}
	Let $X$ be a weakly special cube complex and $G=\pi_1(X)$. Let $\wt{X}\to X$ be the universal cover, $\wt{A}\subset\wt{X}$ a convex subcomplex and $\ti{x}\in\wt{A}$ a vertex. Let $\wt{B}$ be the orthogonal complement of $\wt{A}$ at $\ti{x}$. Then:
	\begin{enumerate}
		\item $\wt{B}/G_{\wt{B}}$ embeds in $X$.
		\item $G_{\wt{A}}$ and $G_{\wt{B}}$ commute.
		\item $G_{\wt{A}}$ and $G_{\wt{B}}$ preserve $\wt{P}_{\wt{A}}=\phi(\wt{A}\times\wt{B})$, moreover $G_{\wt{A}}$ (resp. $G_{\wt{B}}$) acts on the $\wt{A}$ factor (resp. $\wt{B}$ factor) in the natural way and it acts trivially on the other factor.
	\end{enumerate}
\end{prop:orthcomp}

We remark that Proposition \ref{prop:orthcomp} is similar to \cite[Corollary 2.13]{Fioravanti23} in the case where $G_{\wt{A}}$ acts on $\wt{A}$ cocompactly.
We also note that a cube complex being virtually weakly special is a weaker notion than being virtually special (see Section \ref{subsec:directlyspecial}), so Propositions \ref{prop:projection} and \ref{prop:orthcomp} are actually stronger than we need for the proof of Theorem \ref{thm:routesext}.

\subsection{Structure of the paper}

In Section \ref{sec:prelims} we provide some preliminaries about cube complexes and separable subsets of groups.

In Section \ref{sec:projections} we study projection maps and orthogonal complements in CAT(0) cube complexes. This is where we prove Propositions \ref{prop:projection} and \ref{prop:orthcomp}.

In Section \ref{sec:routes} we introduce the notion of routes in non-positively curved cube complexes and we establish some elementary lemmas about them. This is also where we prove Proposition \ref{prop:routesep}, which will allow us to deduce Theorem \ref{thm:prodsep} from Theorem \ref{thm:routesext}.

In Section \ref{sec:control} we begin the proof of Theorem \ref{thm:embroutes} (which is a weaker version of Theorem \ref{thm:routesext}). This section comprises a number of lemmas which gradually reduce the theorem to the case where the route $\fR$ satisfies certain properties (which are needed in Section \ref{sec:walker}).

In Section \ref{sec:walker} we recall the Walker and Imitator construction, which, together with the properties granted by Section \ref{sec:control}, allows us to complete the proof of Theorem \ref{thm:embroutes}. 

In Section \ref{sec:proof} we deduce Theorems \ref{thm:routesext} and \ref{thm:prodsep} from Theorem \ref{thm:embroutes}.

In Section \ref{sec:actions} we prove Proposition \ref{prop:ballsep} and Corollary \ref{cor:contact} regarding separable subsets in group actions.

\textbf{Acknowledgements:}\,
The author would like to thank Lawk Mineh, Mark Hagen and Motiejus Valiunas for their helpful comments. The author is also grateful for the comments of the anonymous referee, which in particular led to Corollary \ref{cor:relhyp} being generalized.

\bigskip
\section{Preliminaries}\label{sec:prelims}

See \cite{HaglundWise08} for the definitions of cube complex, hyperplane and non-positively curved cube complex. In this section we recall some more specialised definitions regarding cube complexes and specialness. We also recall some facts about separable subsets of groups at the end of the section. First some notational conventions.

\begin{nota}
	Let $X$ be a cube complex. 
	By a \emph{path} in $X$ we will always mean an edge path, and a \emph{loop} or \emph{closed path} is a path that starts and finishes at the same vertex.
	In this paper all cube complexes and their covers will be connected unless otherwise stated.
\end{nota}

\subsection{Local isometries and elevations}

\begin{defn}(Local isometries)\\\label{defn:locisom}
	A \emph{local isometry} $\phi:Y\to X$ of non-positively curved cube complexes is a combinatorial map such that each induced map on links of vertices $\link(y)\to\link(\phi(y))$ is an embedding with image a full subcomplex of $\link(\phi(y))$ (a subcomplex $C$ of a simplicial complex $D$ is \emph{full} if any simplex of $D$ whose vertices are in $C$ is in fact entirely contained in $C$). A subcomplex $Y\subset X$ is \emph{locally convex} if the inclusion $Y\xhookrightarrow{}X$ is a local isometry.	
	If $X$ is also CAT(0) then we simply say that $Y$ is \emph{convex} (and this is equivalent to $Y$ being convex in the combinatorial or CAT(0) metrics on $X$).
\end{defn}

\begin{remk}\label{remk:isembedding}
	If $\phi:Y\to X$ is a local isometry with $X$ CAT(0), then $Y$ is also CAT(0) and $\phi$ is an isometric embedding, so $Y$ can be identified with a convex subcomplex of $X$ \cite[Proposition II.4.14]{BridsonHaefliger99}.
\end{remk}

\begin{defn}(Elevations)\\\label{defn:elevations}
	Let $\phi:Y\to X$ be a local isometry of non-positively curved cube complexes and $\mu:\hat{X}\to X$ a cover. We say that a map $\hat{\phi}:\hat{Y}\to\hat{X}$ is an \emph{elevation} of $\phi$ to $\hat{X}$ if there exists a covering $\nu:\hat{Y}\to Y$ fitting into the commutative diagram
	\begin{equation}\label{elevation}
		\begin{tikzcd}[
			ar symbol/.style = {draw=none,"#1" description,sloped},
			isomorphic/.style = {ar symbol={\cong}},
			equals/.style = {ar symbol={=}},
			subset/.style = {ar symbol={\subset}}
			]
			\hat{Y}\ar{d}{\nu}\ar{r}{\hat{\phi}}&\hat{X}\ar{d}{\mu}\\
			Y\ar{r}{\phi}&X,
		\end{tikzcd}
	\end{equation}
	and such that a path in $\hat{Y}$ closes up as a loop if and only if its projections to $Y$ and $\hat{X}$ both close up as loops. The map $\hat{\phi}$ is necessarily a local isometry. Equivalently, $\hat{Y}$ is a component of the pullback of $\phi$ and $\mu$. 
	Often $Y$ will be a subcomplex of $X$, in which case the elevations of $Y$ will simply be the components of the preimage of $Y$ in $\hat{X}$.
\end{defn}

\begin{remk}\label{remk:univelev}
	If $\hat{X}=\wt{X}$ is the universal cover of $X$, then $\hat{Y}=\wt{Y}$ must be simply connected and $\hat{\phi}:\wt{Y}\to\wt{X}$ is an isometric embedding of CAT(0) cube complexes (see Remark \ref{remk:isembedding}). In this case we can consider $\wt{Y}$ to be a convex subcomplex of $\wt{X}$.
	
	If $G=\pi_1(X)$, then we have an action of $G$ on $\wt{X}$ by deck transformations.
	If we fix a map $\phi:Y\to X$ as in Definition \ref{defn:elevations}, then $G$ acts on the set of elevations of $\phi$ to $\wt{X}$, i.e., if $\wt{Y}$ is one elevation of $\phi$ to $\wt{X}$ and $g\in G$, then $g\wt{Y}$ is another elevation of $\phi$ to $\wt{X}$. Moreover, the covering map $g\wt{Y}\to Y$ is just the composition $g\wt{Y}\overset{g^{-1}}{\longrightarrow}\wt{Y}\to Y$.
	Additionally, since $G$ acts transitively on each fibre of $\wt{X}\to X$, the action of $G$ on the set of elevations of $\phi$ is transitive.
\end{remk}

\begin{defn}(Based elevations)\\\label{defn:basedelevations}
	Let $\phi:(Y,y)\to (X,x)$ be a based local isometry of non-positively curved cube complexes and $\mu:(\hat{X},\hat{x})\to (X,x)$ a based cover. We say that a based map $\hat{\phi}:(\hat{Y},\hat{y})\to(\hat{X},\hat{x})$ is the \emph{based elevation} of $\phi$ to $(\hat{X},\hat{x})$ if there exists a based covering $\nu:(\hat{Y},\hat{y})\to (Y,y)$ fitting into the commutative diagram
	\begin{equation}\label{belevation}
		\begin{tikzcd}[
			ar symbol/.style = {draw=none,"#1" description,sloped},
			isomorphic/.style = {ar symbol={\cong}},
			equals/.style = {ar symbol={=}},
			subset/.style = {ar symbol={\subset}}
			]
			(\hat{Y},\hat{y})\ar{d}{\nu}\ar{r}{\hat{\phi}}&(\hat{X},\hat{x})\ar{d}{\mu}\\
			(Y,y)\ar{r}{\phi}&(X,x),
		\end{tikzcd}
	\end{equation}
	and such that a path in $\hat{Y}$ closes up as a loop if and only if its projections to $Y$ and $\hat{X}$ both close up as loops. Equivalently, $\hat{Y}$ is the component of the pullback of $\phi$ and $\mu$ that contains the vertex $(y,\hat{x})$ (and this vertex is identified with $\hat{y}$) -- this perspective also shows why the based elevation is unique.
	If we do not wish to prescribe the basepoint $\hat{x}\in\hat{X}$, then we can also refer to $\hat{\phi}:(\hat{Y},\hat{y})\to(\hat{X},\hat{x})$ as a \emph{based elevation} of $\phi$ to $\hat{X}$.
\end{defn}

\begin{remk}\label{remk:unibasedelev}
	Similarly to Remark \ref{remk:univelev}, if $\hat{X}=\wt{X}$ is the universal cover of $X$, then based elevations can be considered as based convex subcomplexes of $\wt{X}$.
	And if we fix a based map $\phi:(Y,y)\to (X,x)$ as in Definition \ref{defn:basedelevations}, then $G=\pi_1(X)$ acts transitively on the set of based elevations of $\phi$ to $\wt{X}$.
\end{remk}

\subsection{Special cube complexes}\label{subsec:directlyspecial}

\begin{defn}(Parallelism)\\\label{defn:parallelism}
	Let $X$ be a cube complex. Two edges $e_1,e_2$ in $X$ are \emph{elementary parallel} if they appear as opposite edges of some square of $X$. The relation of elementary parallelism generates the equivalence relation of \emph{parallelism}. We write $e_1\parallel e_2$ if edges $e_1$ and $e_2$ are parallel. We define $H(e_1)$ to be the hyperplane dual to an edge $e_1$ -- note that $e_1\parallel e_2$ is equivalent to $H(e_1)=H(e_2)$.
\end{defn}

\begin{defn}(Intersecting and osculating hyperplanes)\\
	Let $X$ be a cube complex. Suppose distinct edges $e_1$ and $e_2$ of $X$ are incident at a vertex $x$.
	\begin{enumerate}
		\item If $e_1$ and $e_2$ form the corner of a square at $x$, then we say that the hyperplanes $H(e_1)$ and $H(e_2)$ \emph{intersect at $(x;e_1,e_2)$}. If in addition $H(e_1)=H(e_2)$, then we say that $H(e_1)$ \emph{self-intersects at $(x;e_1,e_2)$}. Note that a pair of hyperplanes intersect (as subsets of $X$) if and only if they are equal or they intersect at some $(x;e_1,e_2)$.
		\item If $e_1$ and $e_2$ do not form the corner of a square at $x$, then we say that the hyperplanes $H(e_1)$ and $H(e_2)$ \emph{osculate at $(x;e_1,e_2)$}. If in addition $H(e_1)=H(e_2)$, then we say that $H(e_1)$ \emph{self-osculates at $(x;e_1,e_2)$}. Alternatively, if $e_1$ has both its ends incident at $x$, then we say that $H(e_1)$ \emph{self-osculates at $(x;e_1)$}. We say that a pair of hyperplanes \emph{osculate} if they osculate at some $(x;e_1,e_2)$. 
		\item We say that distinct hyperplanes $H_1$ and $H_2$ \emph{inter-osculate} if they both intersect and osculate. 
	\end{enumerate} 
	We will sometimes just say that a pair of hyperplanes intersect or osculate at a vertex $x$ if we do not wish to specify edges $e_1$ and $e_2$.
	The notation from \cite{HaglundWise08} is slightly different from ours as they work with oriented edges and distinguish between direct and indirect self-osculations.
\end{defn}

\begin{defn}(Two-sided hyperplanes)\\\label{defn:twosided}
	A hyperplane $H$ is \emph{two-sided} if the map $H\to X$ extends to a combinatorial map $H\times[-1,1]\to X$ (where we consider $H$ with its induced cube complex structure).
\end{defn}

\begin{defn}(Right-angled Artin groups and Salvetti complexes)\\\label{defn:RAAG}
	Let $\Gamma$ be a simplicial graph (possibly infinite). The \emph{right-angled Artin group} associated to $\Gamma$ is the group given by the following presentation
	$$A_\Gamma=\langle V\Gamma\mid [u,v]:\,(u,v)\in E\Gamma\rangle.$$
	The \emph{Salvetti complex} associated to $\Gamma$, is a non-positively curved cube complex $X_\Gamma$ with fundamental group $A_\Gamma$, constructed as follows.
	The 1-skeleton of $X_\Gamma$ has a single vertex and, for each $v\in V\Gamma$, an oriented edge labelled by $v$.
	Inductively attach higher dimensional cubes as follows.
	For each finite complete subgraph of $\Gamma$ spanned by vertices $\{v_1,\dots,v_n\}$,
	take an $n$-cube $C$, and make a correspondence between the vertices $v_i$ and the parallelism classes of edges in $C$. For the parallelism class corresponding to $v_i$, orient the edges consistently and label them by $v_i$. Then glue each opposite pair of codimension-1 faces of $C$ to the unique $(n-1)$-cube in $X_\Gamma$ with the same edge labels.
	Note that the 2-skeleton of $X_\Gamma$ is the presentation complex of $A_\Gamma$ (with respect to the above presentation).
\end{defn}

\begin{defn}(Notions of specialness)\\\label{defn:special}
	Let $X$ be a non-positively curved cube complex.
	\begin{itemize}
		\item We say that $X$ is \emph{special} if it admits a local isometry to some Salvetti complex of a right-angled Artin group, $X\to X_\Gamma$.
		\item We say that $X$ is \emph{weakly special} if no hyperplane self-intersects or self-osculates.
		\item We say that $X$ is \emph{directly special} if every hyperplane is two-sided, no hyperplane self-intersects or self-osculates, and no pair of hyperplanes inter-osculate.
		\item We say that $X$ is \emph{virtually special} if it has a finite-sheeted special cover.
	\end{itemize}
\end{defn}

\begin{remk}\label{remk:special}
	\begin{itemize}
		\item The definition of special cube complex can be equivalently formulated in terms of intersections and osculations of hyperplanes. In particular, no hyperplane in a special cube complex self-intersects. However, the self-osculation condition requires working with oriented hyperplanes, which will not concern us in this paper. The definition originally appeared in \cite{HaglundWise08} under the name \emph{$A$-special}.
		\item The definition of weakly special cube complex first appeared in \cite{NakagawaTamuraYamashita14} and it also appears in \cite{Huang17,Oh22}.
		\item The definition of directly special cube complex appears in \cite{Huang18}, and directly special implies both special and weakly special.
		\item Among finite non-positively curved cube complexes, the notions of special and directly special are equivalent up to finite covers \cite[Proposition 3.10]{HaglundWise08}. However, it is unknown whether a finite weakly special cube complex is always virtually special (this is almost the same as \cite[Problem 11.1]{HaglundWise08}).
	\end{itemize}
\end{remk}

\subsection{Some more about hyperplanes}

\begin{defn}(Hyperplane carriers)\\\label{defn:carrier}
	Let $H$ be a hyperplane in a directly special cube complex $X$. The \emph{carrier of $H$}, denoted $N(H)$, is the smallest subcomplex of $X$ containing $H$. 
	In fact, the map $H\times[-1,1]\to X$ from Definition \ref{defn:twosided} is an embedding with image $N(H)$.
	Hyperplane carriers in directly special cube complexes are always locally convex subcomplexes.
	If $X$ is not directly special, one can still define the carrier of $H$ as a certain local isometry $N(H)\to X$ (see \cite{WiseRiches}), but this is more delicate and will not concern us.
\end{defn}

\begin{remk}\label{remk:carrierelev}
	If $H$ is a hyperplane in a directly special cube complex $X$, and $\hat{X}\to X$ is a finite cover, then each hyperplane $\hat{H}$ in $\hat{X}$ which maps to $H$ is in fact a cover of $H$.
	From this, one can deduce that we also get a covering map of hyperplane carriers $N(\hat{H})\to N(H)$.
	In particular, $N(\hat{H})$ is an elevation of $N(H)$, and in fact every elevation of $N(H)$ to $\hat{X}$ is of this form.
\end{remk}

The notion of two hyperplanes in a cube complex $X$ intersecting or osculating generalizes to the notion of a hyperplane and a subcomplex intersecting or osculating as follows. This generalization first appeared in \cite[Remark A.9]{OregonReyes23}.

\begin{defn}(Intersections and osculations of hyperplanes with subcomplexes)\\\label{defn:complexosculate}
	Let $Y$ be a locally convex subcomplex of a non-positively curved cube complex $X$. Let $H$ be a hyperplane in $X$.
	\begin{enumerate}
		\item We say that $H$ and $Y$ \emph{osculate at $(y;e)$} if $y\in Y$ is a vertex, $e$ is an edge in $X$ incident to $y$ whose interior lies outside of $Y$, and $H=H(e)$.
		\item We say that $H$ and $Y$ \emph{inter-osculate} if they both intersect and osculate.
	\end{enumerate}
\end{defn}

With the following proposition we can often reduce to working with subcomplexes instead of local isometries, and we can avoid inter-osculations between hyperplanes and subcomplexes.

\begin{prop}\cite[Corollary 4.8]{Shepherd23}\\\label{prop:elevembed}
	Let $Y_1,...,Y_n\to X$ be local isometries of finite virtually special cube complexes. Then there is a finite directly special regular cover $\hat{X}\to X$ such that all elevations of the $Y_i$ to $\hat{X}$ are embedded and do not inter-osculate with hyperplanes of $\hat{X}$.
\end{prop}

\subsection{Virtually special actions and convex-cocompact subgroups}

Let us now recall the definitions of virtually special actions and convex-cocompact subgroups from the introduction.

\begin{defn}(Virtually special action)\\\label{defn:cubulated}
	An action $G\acts \wt{X}$ on a CAT(0) cube complex is \emph{virtually special} if there exists a finite-index subgroup $G'<G$ acting freely on $\wt{X}$, with the quotient $\wt{X}/G'$ a special cube complex.		
	Note that the fundamental group of a (virtually special) non-positively curved cube complex gives rise to a (virtually special) free action on a CAT(0) cube complex by considering the action on the universal cover by deck transformations.
\end{defn}

\begin{defn}(Convex-cocompact subgroups)\\\label{defn:convexsubgp}
	Let $G\acts \wt{X}$ be a proper action on a CAT(0) cube complex. A subgroup $K<G$ is \emph{convex-cocompact} if it stabilizes and acts cocompactly on a convex subcomplex $\wt{Y}\subset \wt{X}$.
\end{defn}

\begin{remk}\label{remk:convexsubgp}
	If $G=\pi_1(X,x)$ is the fundamental group of a non-positively curved cube complex $X$, then $K<G$ being convex-cocompact is equivalent to $K$ being the image of a homomorphism $\phi_*:\pi_1(Y,y)\xhookrightarrow{}G$, defined by a local isometry $\phi:Y\to X$, with $Y$ finite, and (a homotopy class of) a path $\gamma$ in $X$ from $x$ to $\phi(y)$. We recover the first definition with respect to a based universal cover $(\widetilde{X},\ti{x})\to (X,x)$ by letting $\ti{\gamma}$ be a lift of $\gamma$ from $\ti{x}$ to $\ti{y}$, and setting $\widetilde{Y}$ to be the based elevation of $\phi:Y\to X$ with respect to basepoints $y,\phi(y)$ and $\ti{y}$ (note that $\widetilde{Y}$ is a convex subcomplex of $\widetilde{X}$ by Remark \ref{remk:univelev}). The other elevations of $Y$ are stabilized by conjugates of $K$.
\end{remk}

\subsection{Separability}\label{subsec:separability}

Recall from the introduction that a subset $S$ of a group $G$ is \emph{separable} if, for any $g\in G- S$, there is a homomorphism $\phi:G\to F$ to a finite group $F$ with $\phi(g)\notin\phi(S)$.
This has the following well-known equivalent formulations.

\begin{lem}\label{lem:sepequiv}
	Let $S$ be a subset of a group $G$. The following are equivalent:
	\begin{enumerate}
		\item $S$ is separable in $G$.
		\item For any $g\in G- S$, there exists a finite-index normal subgroup $\hat{G}\triangleleft G$ with $g\notin\hat{G}S$.
		\item\label{it:sep3} For any $g\in G- S$, there exists a finite-index subgroup $\hat{G}<G$ with $g\notin\hat{G}S$.
	\end{enumerate}
\end{lem}

The following lemma allows us to pass to finite-index subgroups when proving Theorem \ref{thm:prodsep}.

\begin{lem}\label{lem:sepfiniteindex}
	Let $G\acts X$ be a proper action on a CAT(0) cube complex and let $\hat{G}\triangleleft G$ be a finite-index normal subgroup.
	If every product of convex-cocompact subgroups of $\hat{G}$ is separable in $\hat{G}$ (with respect to the induced action $\hat{G}\acts X$), then every product of convex-cocompact subgroups of $G$ is separable in $G$.
\end{lem}
\begin{proof}
Let $K_1,K_2,\dots,K_n<G$ be convex-cocompact subgroups. Let $\hat{K}_i=K_i\cap\hat{G}$ for $1\leq i\leq n$. The product $K_1K_2\cdots K_n$ can be written as a finite union of multi cosets $g_1\hat{K}_1g_2\hat{K}_2\cdots g_n\hat{K}_n$, with $g_1,\dots,g_n\in G$.
Putting $h_i=g_1g_2\cdots g_i$ for $1\leq i\leq n$, the multi coset above can be rewritten
as
\begin{equation}\label{multicoset}
	h_1\hat{K}_1h_1^{-1}h_2\hat{K}_2h_2^{-1}\cdots h_n\hat{K}_nh_n^{-1}h_n.
\end{equation}
Each $h_i\hat{K}_ih_i^{-1}$ is a convex-cocompact subgroup of $\hat{G}$, so their product is separable in $\hat{G}$, hence also separable in $G$. Separability of a subset is preserved by left and right multiplication, so the subset (\ref{multicoset}) is separable in $G$. Finite unions of separable subsets are separable, thus the original product $K_1K_2\cdots K_n$ is separable in $G$.
\end{proof}

\bigskip
\section{Projections and orthogonal complements}\label{sec:projections}

In this section we study projection maps and orthogonal complements in CAT(0) cube complexes. We will be especially interested in the case where we have a covering map from a CAT(0) cube complex to a weakly special cube complex. In particular, we will prove Propositions \ref{prop:projection} and \ref{prop:orthcomp}, which will be needed in Section \ref{sec:control}.
We start with some basic definitions and lemmas about convex hulls and convex subcomplexes.

\subsection{Convex hulls and convex subcomplexes}\label{subsec:conhulls}

\begin{defn}(Transverse hyperplanes and convex hulls)\\\label{defn:convexhulls}
Let $\wt{X}$ be a CAT(0) cube complex.
We say that a pair of hyperplanes $\wt{H}_1,\wt{H}_2$ in $\wt{X}$ are \emph{transverse} if they are distinct and have non-empty intersection.
Following \cite{CapraceSageev11}, if $\wt{A}\subset\wt{X}$ is a convex subcomplex, we write $\hH(\wt{A})$ for the set of hyperplanes of $\wt{X}$ that intersect $\wt{A}$ (which is in natural bijection with the set of hyperplanes of $\wt{A}$).
If $S\subset\wt{X}$ is any subcomplex (even disconnected), we define the \emph{convex hull} of $S$, denoted $\Hull(S)$, to be the unique smallest convex subcomplex of $\wt{X}$ that contains $S$ (in fact $\Hull(S)$ is the intersection of all the convex subcomplexes that contain $S$).
Since each hyperplane separates $\wt{X}$ into two convex subcomplexes (\emph{halfspaces}), it follows that $\hH(\Hull(S))$ is precisely the set of hyperplanes that separate a pair of vertices in $S$.
\end{defn}

\begin{lem}\label{lem:pathunion}
Let $\wt{X}$ be a CAT(0) cube complex and let $\wt{A}\subset\wt{X}$ be a convex subcomplex.
If $\ti{a}\in\wt{A}$ is a vertex, then the 1-skeleton of $\wt{A}$ is the union of all paths $\gamma$ that start at $\ti{a}$ and which only cross hyperplanes in $\hH(\wt{A})$.
Furthermore, $\wt{A}$ is the union of all cubes in $\wt{X}$ whose 1-skeleta are contained in the 1-skeleton of $\wt{A}$.
\end{lem}
\begin{proof}
As $\wt{A}$ is convex, it is an intersection of halfspaces in $\wt{X}$. None of the halfspaces in this intersection are bounded by hyperplanes in $\hH(\wt{A})$, so if an edge in $\wt{X}$ dual to a hyperplane in $\hH(\wt{A})$ has one endpoint in $\wt{A}$, then the whole edge must be in $\wt{A}$. The first part of the lemma follows.
The second part of the lemma holds because one can determine which cubes of dimension 2 and above are in $\wt{A}$ by examining how the 1-skeleton of $\wt{A}$ intersects the links of vertices (see Definition \ref{defn:locisom}).
\end{proof}

\begin{lem}\label{lem:stabilize}
Let $\wt{X}$ be a CAT(0) cube complex and let $\wt{A}\subset\wt{X}$ be a convex subcomplex. Suppose $g$ is an automorphism of $\wt{X}$ that stabilizes $\hH(\wt{A})$ and suppose there is a vertex $\ti{a}\in\wt{A}$ with $g\ti{a}\in\wt{A}$.
Then $g$ stabilizes $\wt{A}$.
\end{lem}
\begin{proof}	
	For each path $\gamma$ that starts at $\ti{a}$ and which only crosses hyperplanes in $\hH(\wt{A})$, the image $g\gamma$ will still only cross hyperplanes in $\hH(\wt{A})$. There exists a path from $\ti{a}$ to $g\ti{a}$ that only crosses hyperplanes in $\hH(\wt{A})$, hence $g\wt{A}\subset\wt{A}$ by Lemma \ref{lem:pathunion}.
	 Applying the same argument to $g^{-1}$, we conclude that $g\wt{A}=\wt{A}$.
\end{proof}

The next two lemmas will be needed in Section \ref{sec:routes}.

\begin{lem}\label{lem:vspecfdim}
If a special cube complex $X$ has finitely many hyperplanes, then $X$ is finite dimensional and locally finite.
\end{lem}
\begin{proof}
	Let $X\to A_\Gamma$ be a local isometry to the Salvetti complex of a right-angled Artin group. Since $X$ has finitely many hyperplanes, there is a finite subgraph $\Lambda\subset\Gamma$ such that the image of $X$ in $X_\Gamma$ is contained in the natural subcomplex $X_\Lambda\subset X_\Gamma$.
	The restriction $X\to X_\Lambda$ is still a local isometry, so the lemma follows from the fact that $X_\Lambda$ is a finite cube complex.
\end{proof}

\begin{lem}\label{lem:vspeccocomp}
Let $G\acts\wt{X}$ be a virtually special action of a group $G$ on a CAT(0) cube complex $\wt{X}$.
Suppose $\wt{Y}\subset\wt{X}$ is a $G$-invariant $G$-cocompact convex subcomplex, and suppose $\ti{x}\in\wt{X}$ is a vertex.
Then the action of $G$ on $\wt{Z}=\Hull(\wt{Y}\cup G\cdot\ti{x})$ is cocompact.
\end{lem}
\begin{proof}
	Passing to a finite-index subgroup if necessary, we may assume that $G$ acts freely on $\wt{X}$ with the quotient $\wt{X}/G$ a special cube complex.
	If $\gamma$ is a path joining $\ti{x}$ to $\wt{Y}$, then $\wt{Y}\cup \bigcup_{g\in G}g\gamma$ is a connected subcomplex of $\wt{X}$ that contains $\wt{Y}\cup G\cdot\ti{x}$ and has finitely many $G$-orbits of hyperplanes.
	It follows that $\wt{Z}=\Hull(\wt{Y}\cup G\cdot\ti{x})$ has finitely many $G$-orbits of hyperplanes.
	Since $\wt{Z}/G$ is special ($\wt{Z}/G\to\wt{X}/G$ is a local isometry), it follows
	from Lemma \ref{lem:vspecfdim} that $\wt{Z}$ is finite dimensional and locally finite.
	
	We claim that $\wt{Z}$ is contained in a bounded neighbourhood of $\wt{Y}$.
	Since $\wt{Z}$ is convex and locally finite, and since $\wt{Y}$ is $G$-cocompact, it will follow from the claim that $\wt{Z}$ is $G$-cocompact.
	Let $\ti{z}\in\wt{Z}$ be a vertex and let $\cA$ denote the set of hyperplanes of $\wt{X}$ that separate $\ti{z}$ from $\wt{Y}$.
	Suppose $\wt{H}_1,\wt{H}_2\in\cA$ are distinct hyperplanes in the same $G$-orbit.
	Since $\wt{Z}/G$ is special, $\wt{H}_1$ and $\wt{H}_2$ must be disjoint.
	Then (up to swapping indices) $\wt{H}_1$ separates $\wt{H}_2$ from $\wt{Y}$, so $\wt{Y}$ is closer to $\wt{H}_1$ than $\wt{H}_2$.
	But this is impossible since $\wt{Y}$ is $G$-invariant and $\wt{H}_1$ and $\wt{H}_2$ are in the same $G$-orbit.
	It follows that the hyperplanes in $\cA$ are in different $G$-orbits, which gives us an upper bound on $|\cA|$ independent of $\ti{z}$.
	Since $|\cA|$ is equal to the combinatorial distance from $\ti{z}$ to $\wt{Y}$, this proves the claim that $\wt{Z}$ is contained in a bounded neighbourhood of $\wt{Y}$.	
\end{proof}

\subsection{Projections and orthogonal subcomplexes}

\begin{defn}(Orthogonal subcomplexes)\\\label{defn:orthogonal}
Let $\wt{X}$ be a CAT(0) cube complex and let $\wt{A},\wt{B}\subset\wt{X}$ be convex subcomplexes. We say that $\wt{A}$ and $\wt{B}$ are \emph{orthogonal} if their intersection is a single vertex $\tilde{x}$ and if every hyperplane in $\hH(\wt{A})$ is transverse to every hyperplane in $\hH(\wt{B})$.
In this case, there is an isomorphism of CAT(0) cube complexes $\phi: \wt{A}\times\wt{B}\to\Hull(\wt{A}\cup\wt{B})$, with $\phi(\tilde{x},\tilde{x})=\tilde{x}$ and the restrictions $\phi:\wt{A}\times\{\tilde{x}\}\to\wt{A}$ and $\phi:\{\tilde{x}\}\times\wt{B}\to\wt{B}$ induced by the identity maps on $\wt{A}$ and $\wt{B}$ (this follows from \cite[Lemma 2.5]{CapraceSageev11} and the fact that $\hH(\Hull(\wt{A}\cup\wt{B}))=\hH(\wt{A})\cup\hH(\wt{B})$).

\end{defn}

\begin{defn}\label{defn:projection}(Projection maps)\\
	Let $\wt{X}$ be a CAT(0) cube complex and let $\wt{A}\subset\wt{X}$ be a convex subcomplex. The \emph{projection to $\wt{A}$} is the combinatorial map $\Pi_{\wt{A}}:\wt{X}\to \wt{A}$ that sends each vertex $\tilde{x}\in \wt{X}$ to the unique closest vertex in $\wt{A}$ with respect to the combinatorial metric -- this is well-defined by \cite[Lemma 13.8]{HaglundWise08}.
\end{defn}

These projection maps can be used to construct ``bridges'' between convex subcomplexes as follows. This theorem appears in lecture notes of Hagen as \cite[Theorem 1.22]{Hagen19} (see also \cite{ChatterjiFernosIozzi16} and \cite{Shepherd23}).

\begin{thm}(Bridge Theorem)\\\label{thm:bridge}
	Let $\wt{X}$ be a CAT(0) cube complex and let $\wt{A},\wt{B}\subset\wt{X}$ be convex subcomplexes.
	Let $\wt{C}=\Pi_{\wt{A}}(\wt{B})$.
	Then $\hH(\wt{C})=\hH(\wt{A})\cap\hH(\wt{B})$.
	Furthermore, there is a convex subcomplex $\wt{D}$ orthogonal to $\wt{C}$, and vertices $\tilde{a}\in\wt{D}\cap\wt{C}$ and $\ti{b}\in\wt{D}\cap\wt{B}$ such that the isomorphism $\phi:\wt{C}\times\wt{D}\to\Hull(\wt{C}\cup\wt{D})$ satisfies:
	\begin{enumerate}
		\item\label{item:intA} $\phi(\wt{C}\times\wt{D})\cap\wt{A}=\phi(\wt{C}\times\{\tilde{a}\})=\wt{C}$,
		\item\label{item:intB} $\phi(\wt{C}\times\wt{D})\cap\wt{B}=\phi(\wt{C}\times\{\tilde{b}\})=\Pi_{\wt{B}}(\wt{A})$, and
		\item\label{item:sephyp} $\hH(\wt{D})$ is precisely the set of hyperplanes that separate $\wt{A}$ and $\wt{B}$.
	\end{enumerate}  
\end{thm}

First, we use the Bridge Theorem to prove the following lemma about elevations of locally convex subcomplexes.

\begin{lem}\label{lem:hH=hH}
	Let $X$ be a finite non-positively curved cube complex with universal cover $\wt{X}$, and let $Z\subset X$ be a locally convex subcomplex.
	If $\wt{A},\wt{B}$ are two elevations of $Z$ to $\wt{X}$ with $\hH(\wt{B})\subset\hH(\wt{A})$, then $\hH(\wt{B})=\hH(\wt{A})$.
\end{lem}
\begin{proof}
	First, we show that the projection map $\Pi_{\wt{A}}:\wt{B}\to\wt{A}$ is a combinatorial embedding (i.e. it identifies $\wt{B}$ with a subcomplex of $\wt{A}$).
	Apply the Bridge Theorem to $\wt{A}$ and $\wt{B}$. Since $\hH(\wt{B})\subset\hH(\wt{A})$, we have $\hH(\wt{C})=\hH(\wt{B})$.
	Then $\phi(\wt{C}\times\wt{D})\cap\wt{B}=\phi(\wt{C}\times\{\tilde{b}\})$ is a convex subcomplex of $\wt{B}$ with $\hH(\phi(\wt{C}\times\{\tilde{b}\}))=\hH(\wt{C})=\hH(\wt{B})$.
	By Lemma \ref{lem:pathunion}, $\wt{B}=\phi(\wt{C}\times\{\tilde{b}\})$.
	Now, for each vertex $\ti{c}\in\wt{C}$, Theorem \ref{thm:bridge}\ref{item:sephyp} implies that the set of hyperplanes separating $\phi(\ti{c},\ti{a})$ from $\phi(\ti{c},\ti{b})$ is the same as the set of hyperplanes separating $\wt{A}$ from $\wt{B}$, and is equal to $\hH(\wt{D})$.
	Since the combinatorial distance between a pair of vertices in a CAT(0) cube complex is equal to the number of hyperplanes that separate them, we deduce that $\Pi_{\wt{A}}(\ti{c},\ti{b})=(\ti{c},\ti{a})$.
	This holds for all $\ti{c}\in\wt{C}$, so $\Pi_{\wt{A}}$ restricts to an isomorphism $\wt{B}=\phi(\wt{C}\times\{\tilde{b}\})\to \wt{C}=\phi(\wt{C}\times\{\tilde{a}\})$, which is a combinatorial embedding into $\wt{A}$. 

	Now we show that the projection map $\Pi_{\wt{A}}:\wt{B}\to\wt{A}$ is actually an isomorphism, from which it follows that $\hH(\wt{A})\subset\hH(\wt{B})$ (again by the Bridge Theorem), so $\hH(\wt{B})=\hH(\wt{A})$.
	For a vertex $\tilde{b}\in\wt{B}$ and an integer $R\geq0$, let $f_{\wt{B}}(\tilde{b},R)$ denote the number of vertices in the ball of radius $R$ in $\wt{B}$ about $\tilde{b}$  (working in the combinatorial metric of $\wt{B}$ -- which coincides with the combinatorial metric of $\wt{X}$ since $\wt{B}$ is convex).
	Similarly, let $f_{\wt{A}}(\tilde{a},R)$ denote the number of vertices in the ball of radius $R$ in $\wt{A}$ about a vertex $\tilde{a}$.
	For two sequences of integers $(\alpha_R)_{R\geq0}$ and $(\beta_R)_{R\geq0}$, write $(\alpha_R)_{R\geq0}\leq(\beta_R)_{R\geq0}$ if $\alpha_R\leq \beta_R$ for all $R$.
	This defines a partial order on the set of integer sequences.
	Since $\Pi_{\wt{A}}:\wt{B}\to\wt{A}$ is a combinatorial embedding, for any vertex $\tilde{b}\in\wt{B}$ we have
	\begin{equation}\label{seqcompare}
(f_{\wt{B}}(\tilde{b},R))_{R\geq0}\leq(f_{\wt{A}}(\Pi_{\wt{A}}(\tilde{b}),R))_{R\geq0},
	\end{equation}
with equality if and only if $\Pi_{\wt{A}}(\wt{B})=\wt{A}$.
	But $\wt{A},\wt{B}$ are both elevations of $Z$, so they are isometric (they can both be identified with the universal cover of $Z$), so the set of sequences $(f_{\wt{A}}(\tilde{a},R))_{R\geq0}$ is the same as the set of sequences $(f_{\wt{B}}(\tilde{b},R))_{R\geq0}$.
	Moreover, since $\wt{A}$ and $\wt{B}$ are cocompact, this is a finite set of sequences.
	Hence, we may take $\ti{b}\in\wt{B}$ with $(f_{\wt{B}}(\tilde{b},R))_{R\geq0}$ a $\leq$-maximal sequence, and (\ref{seqcompare}) is forced to be an equality.
	Therefore $\Pi_{\wt{A}}(\wt{B})=\wt{A}$, as required.
\end{proof}

\begin{nota}
	If $X$ is a non-positively curved cube complex with fundamental group $G$, and $\wt{X}\to X$ is the universal cover, then we have an action by deck transformations $G\acts\wt{X}$. If $\wt{A}\subset\wt{X}$ is a convex subcomplex then we write $G_{\wt{A}}$ for the $G$-stabilizer of $\wt{A}$.
\end{nota}

The Bridge Theorem has the following consequence for weakly special cube complexes (see Definition \ref{defn:special} for the notion of weakly special).
This proposition is similar to \cite[Proposition 5.3]{Shepherd23}.

\begin{prop}\label{prop:projection}
	Let $A$ and $B$ be locally convex subcomplexes in a weakly special cube complex $X$, and let $G=\pi_1(X)$. Let $\wt{X}\to X$ be the universal cover and let $\wt{A},\wt{B}\subset\wt{X}$ be elevations of $A$ and $B$. Let $\wt{C}=\Pi_{\wt{A}}(\wt{B})$.
	Then $G_{\wt{C}}=G_{\wt{A}}\cap G_{\wt{B}}$ and $\wt{C}/G_{\wt{C}}$ embeds in $X$.	
\end{prop}
\begin{proof}
	First note that for $g\in G_{\wt{A}}\cap G_{\wt{B}}$ we have
	$$g\wt{C}=g\Pi_{\wt{A}}(\wt{B})=\Pi_{g\wt{A}}(g\wt{B})=\Pi_{\wt{A}}(\wt{B})=\wt{C}.$$
	So $G_{\wt{A}}\cap G_{\wt{B}}\subset G_{\wt{C}}$.
	
	Now suppose that vertices $\tilde{x},\tilde{y}\in\wt{C}$ map to the same vertex in $X$. Then there is $g\in G$ with $g\tilde{x}=\tilde{y}$. Since $\tilde{x},\tilde{y}\in\wt{A}$ and $A=\wt{A}/G_{\wt{A}}$ embeds in $X$, we have $g\in G_{\wt{A}}$.
	Now apply the Bridge Theorem (Theorem \ref{thm:bridge}).
	Then $\ti{x}=\phi(\ti{x},\ti{a})$ and $\ti{y}=\phi(\ti{y},\ti{a})$, where $\ti{a}$ is as in Theorem \ref{thm:bridge}.
	Let $\gamma$ be a path in $\wt{D}$ from $\ti{a}$ to $\ti{b}$ (with $\ti{b}$ is as in Theorem \ref{thm:bridge}).
	Then $\phi(\{\ti{x}\}\times\gamma)$ and $\phi(\{\ti{y}\}\times\gamma)$ are paths in $\wt{X}$ starting at $\ti{x}$ and $\ti{y}$ respectively, which cross the same sequence of hyperplanes (see Figure \ref{fig:bridgepaths}). Their images in $X$ are paths with the same start point which again cross the same sequence of hyperplanes.
	Since hyperplanes in $X$ do not self-intersect or self-osculate, we deduce that $\phi(\{\ti{x}\}\times\gamma)$ and $\phi(\{\ti{y}\}\times\gamma)$ descend to the same path in $X$.
	As $g\tilde{x}=\tilde{y}$, we must have that $g\phi(\{\ti{x}\}\times\gamma)=\phi(\{\ti{y}\}\times\gamma)$, so in particular $g\phi(\ti{x},\ti{b})=\phi(\ti{y},\ti{b})$.
	But $\phi(\ti{x},\ti{b}),\phi(\ti{y},\ti{b})\in\wt{B}$, and $B=\wt{B}/G_{\wt{B}}$ embeds in $X$, so $g\in G_{\wt{B}}$.
	Hence $g\in G_{\wt{A}}\cap G_{\wt{B}}\subset G_{\wt{C}}$.
	Therefore $\wt{C}/G_{\wt{C}}$ embeds in $X$.
	
	The argument from the previous paragraph still works if our initial assumption was that $g\in G_{\wt{C}}$, so we also get $G_{\wt{C}}\subset G_{\wt{A}}\cap G_{\wt{B}}$.	
\end{proof}

\begin{figure}[H]
	\centering
	\scalebox{0.8}{
		\begin{tikzpicture}[auto,node distance=2cm,
			thick,every node/.style={circle,draw,font=\small},
			every loop/.style={min distance=2cm},
			hull/.style={draw=none},
			]
			\tikzstyle{label}=[draw=none,font=\huge]
			
			\draw[rounded corners=20pt] (-4, -1) rectangle (0, 5) {};
			\draw[rounded corners=20pt] (6, -1) rectangle (10, 5) {};
			\draw (0,0) rectangle (6,4) {};
			\path (0,1) edge [blue,postaction={decoration={markings,mark=at position 0.8 with {\arrow[blue,line width=1mm]{triangle 60}}},decorate}] (6,1);
			\path (0,3) edge [blue,postaction={decoration={markings,mark=at position 0.8 with {\arrow[blue,line width=1mm]{triangle 60}}},decorate}] (6,3);
			\draw[Green,line width=5pt] (0,0)--(0,4);
			\node[fill] at (0,1){};
			\node[fill] at (0,3){};
			
			\node[label] (A) at (-2,4.2) {$\wt{A}$};
			\node[label] (B) at (8,4.2) {$\wt{B}$};
			\node[label,Green] (C) at (-1,2) {$\wt{C}$};
			\node[label,font=\LARGE] (x) at (-0.5,1) {$\ti{x}$};
			\node[label,font=\LARGE] (y) at (-0.5,3) {$\ti{y}$};
			\node[label,blue,font=\LARGE] (px) at (2.5,1.5) {$\phi(\{\ti{x}\}\times\gamma)$};
			\node[label,blue,font=\LARGE] (py) at (2.5,3.5) {$\phi(\{\ti{y}\}\times\gamma)$};
						
		\end{tikzpicture}
	}
	\caption{The proof of Proposition \ref{prop:projection}.}\label{fig:bridgepaths}
\end{figure}
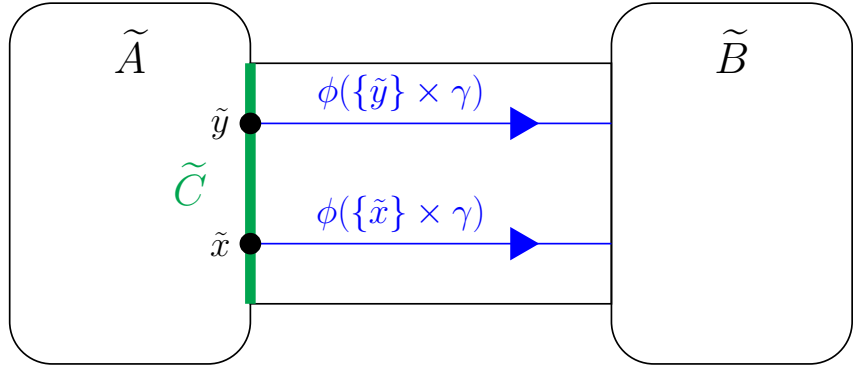

\subsection{Orthogonal complements}

\begin{defn}(Orthogonal complements)\\
Let $\wt{X}$ be a CAT(0) cube complex, let $\wt{A}\subset\wt{X}$ be a convex subcomplex and let $\ti{x}\in\wt{A}$ be a vertex. A convex subcomplex $\wt{A}'$ is \emph{parallel} to $\wt{A}$ if $\hH(\wt{A})=\hH(\wt{A}')$.
Following \cite[Definition 1.10]{HagenSusse20}, we define the \emph{orthogonal complement of $\wt{A}$ at $\ti{x}$} to be the convex subcomplex $\wt{B}$ such that:
\begin{enumerate}
	\item $\ti{x}\in\wt{B}$,
	\item $\wt{B}$ is orthogonal to $\wt{A}$, and
	\item the set of convex subcomplexes parallel to $\wt{A}$ is precisely the collection of slices $\phi(\wt{A}\times\{\ti{b}\})$ (with $\ti{b}\in\wt{B}$ a vertex) coming from the isomorphism $\phi: \wt{A}\times\wt{B}\to\Hull(\wt{A}\cup\wt{B})$.
\end{enumerate} 
Equivalently, one could define $\wt{B}$ as the union of all convex subcomplexes containing $\ti{x}$ that are orthogonal to $\wt{A}$.
We also use the notation $\wt{P}_{\wt{A}}$ for the product subcomplex $\phi( \wt{A}\times\wt{B})=\Hull(\wt{A}\cup\wt{B})$.
\end{defn}

\begin{lem}\cite[Lemma 1.11]{HagenSusse20}\\\label{lem:orthcomphyps}
Let $\wt{X}$ be a CAT(0) cube complex, let $\wt{A}\subset\wt{X}$ be a convex subcomplex and let $\wt{B}$ be the orthogonal complement of $\wt{A}$ at some vertex $\ti{x}\in\wt{A}$.
Then $\hH(\wt{B})$ is the set of hyperplanes in $\wt{X}$ that are transverse to every hyperplane in $\hH(\wt{A})$.
\end{lem}

We have the following proposition about orthogonal complements and weakly special cube complexes. 

\begin{prop}\label{prop:orthcomp}
Let $X$ be a weakly special cube complex and $G=\pi_1(X)$. Let $\wt{X}\to X$ be the universal cover, $\wt{A}\subset\wt{X}$ a convex subcomplex and $\ti{x}\in\wt{A}$ a vertex. Let $\wt{B}$ be the orthogonal complement of $\wt{A}$ at $\ti{x}$. Then:
\begin{enumerate}
	\item\label{item:Bembed} $\wt{B}/G_{\wt{B}}$ embeds in $X$.
	\item\label{item:ABcommute} $G_{\wt{A}}$ and $G_{\wt{B}}$ commute.
	\item\label{item:PApreserved} $G_{\wt{A}}$ and $G_{\wt{B}}$ preserve $\wt{P}_{\wt{A}}=\phi(\wt{A}\times\wt{B})$, moreover $G_{\wt{A}}$ (resp. $G_{\wt{B}}$) acts on the $\wt{A}$ factor (resp. $\wt{B}$ factor) in the natural way and it acts trivially on the other factor.
\end{enumerate}
\end{prop}

\begin{remk}
	If $X$ is finite, then the fact that $G_{\wt{B}}$ acts cocompactly on $\wt{B}$ follows from \cite[Proposition 5.1]{HagenSusse20} because $\wt{X}$ possesses a factor system ($X$ weakly special implies that the action of $G$ on $\wt{X}$ is rotational, so $\wt{X}$ has a factor system by \cite[Theorem A]{HagenSusse20}).
	If in addition $G_{\wt{A}}$ acts on $\wt{A}$ cocompactly, then Proposition \ref{prop:orthcomp} is similar to \cite[Corollary 2.13]{Fioravanti23} (note that $X$ weakly special in Proposition \ref{prop:orthcomp} implies that the action $G\acts\wt{X}$ is non-transverse in the terminology of \cite{Fioravanti23}).
\end{remk}

\begin{proof}[Proof of Proposition \ref{prop:orthcomp}]
	Let $g\in G$ and $\ti{b}\in\wt{B}$ a vertex with $g\ti{b}\in\wt{B}$. Let $\ti{e}_1$ be an edge incident at $\ti{b}$ which is parallel to $\wt{A}$ (i.e. $H(\ti{e}_1)\in\hH(\wt{A})$).
	Since $\wt{B}$ is orthogonal to $\wt{A}$, there is an edge $\ti{e}_2$ incident at $g\ti{b}$ which is parallel to $\ti{e}_1$.
	Either $\ti{e}_2=g\ti{e}_1$, or the hyperplanes $H(\ti{e}_1)=H(\ti{e}_2)$ and $H(g\ti{e}_1)$ intersect or osculate at $(g\ti{b},\ti{e}_2,g\ti{e}_1)$. Since these hyperplanes descend to the same hyperplane in $X$, and since $X$ is weakly special, it must be that $\ti{e}_2=g\ti{e}_1$.
	This argument extends to paths: for any path $\gamma$ that starts at $\ti{b}$ and is parallel to $\wt{A}$, we have that $\gamma$ is parallel to $g\gamma$.
	Therefore, $g$ stabilizes every hyperplane in $\hH(\wt{A})$.
	By Lemma \ref{lem:orthcomphyps}, $\hH(\wt{B})$ is the set of hyperplanes in $\wt{X}$ that are transverse to every hyperplane in $\hH(\wt{A})$, so $g$ also stabilizes the set of hyperplanes $\hH(\wt{B})$.
	It then follows from Lemma \ref{lem:stabilize} that $g$ stabilizes $\wt{B}$.
	This shows that $\wt{B}/G_{\wt{B}}$ embeds in $X$.
	
	Since the element $g\in G$ stabilizes both $\hH(\wt{A})$ and $\hH(\wt{B})$, it must also stabilize $\wt{P}_{\wt{A}}=\phi(\wt{A}\times\wt{B})$ by Lemma \ref{lem:stabilize}, preserving the product decomposition.
	We already established that each path $\gamma$ in $\wt{A}$ is mapped by $g$ to a path parallel to $\gamma$, so $g$ acts trivially on the $\wt{A}$ factor in $\wt{P}_{\wt{A}}=\phi(\wt{A}\times\wt{B})$.
	Since the restriction $\phi:\{\tilde{x}\}\times\wt{B}\to\wt{B}$ is induced by the identity map on  $\wt{B}$, the action of $g$ on the $\wt{B}$ factor is the natural one.
	
	A similar argument can be made, with the roles of $\wt{A}$ and $\wt{B}$ reversed, to show that $G_{\wt{A}}$ preserves $\wt{P}_{\wt{A}}=\phi(\wt{A}\times\wt{B})$, acting on the $\wt{A}$ factor in the natural way and acting trivially on the $\wt{B}$ factor.
	
	We have now proved parts \ref{item:Bembed} and \ref{item:PApreserved} of the lemma. Finally, part \ref{item:ABcommute} follows from part \ref{item:PApreserved} and the fact that $G$ acts freely on $\wt{X}$.	
\end{proof}

\bigskip
\section{Routes}\label{sec:routes}

In this section we introduce the notion of routes in non-positively curved cube complexes, which will be the primary geometric object that we use to study products of convex-cocompact subgroups. We will establish some elementary lemmas about routes and their elevations, and we will prove Proposition \ref{prop:routesep}, which provides the connection between elevations of routes and separability of products of convex-cocompact subgroups.

\subsection{Definitions and basic properties}

\begin{defn}(Routes)\\\label{defn:routes}
	Let $X$ be a non-positively curved cube complex.
	A \emph{route in $X$ of length $n$} is a tuple
	$$\fR=(y_0,Y_1,y_1,Y_2,y_2,\dots,Y_n,y_n),$$
	where $y_i\in X$ are vertices and $Y_i\to X$ are local isometries of non-positively curved cube complexes.
	In addition, each $Y_i$ is equipped with two vertices $y_{i-1}^i,y_i^i\in Y_i$ that map to $y_{i-1},y_i\in X$ respectively.
	(The maps $Y_i\to X$ and the vertices $y_{i-1}^i,y_i^i\in Y_i$ are all part of the data for $\fR$, but we suppress them from the tuple notation for brevity.)
	
	We will often work with $\fR$ an \emph{embedded route}, meaning that the maps $Y_i\to X$ are embeddings. (In this case the $Y_i$ are just locally convex subcomplexes of $X$, so the tuple notation for $\fR$ really does capture all the data.)
	The vertices $y_0$ and $y_n$ are referred to as the \emph{initial vertex of $\fR$} and the \emph{terminal vertex of $\fR$} respectively.
	The route $\fR$ is \emph{closed} if $y_0=y_n$, and it is \emph{finite} if each $Y_i$ is finite.
\end{defn}

\begin{defn}(Segmented paths)\\
	Let $X$ be a non-positively curved cube complex. A \emph{segmented path in $X$ of length $n$} is an $n$-tuple of paths $\delta=(\delta_1,\delta_2,\dots,\delta_n)$, such that the endpoint of $\delta_i$ is the startpoint of $\delta_{i+1}$ ($1\leq i\leq n-1$).
	The \emph{realization} of $\delta$ is the concatenation $\delta_1\delta_2\cdots\delta_n$, which is a path in $X$.
	We say that $\delta$ is \emph{closed} if its realization is closed.
	If $\delta$ is closed, then we say that it is \emph{essential} (resp. \emph{null-homotopic}) if its realization is essential (resp. null-homotopic).
	
	If $\hat{X}\to X$ is a cover, $\delta=(\delta_1,\delta_2,\dots,\delta_n)$ a segmented path in $X$, and $\hat{x}\in \hat{X}$, then a \emph{lift of $\delta$ based at $\hat{x}$} is a segmented path $\hat{\delta}=(\hat{\delta}_1,\hat{\delta}_2,\dots,\hat{\delta}_n)$ such that $\hat{\delta}_1$ is based at $\hat{x}$ and each $\hat{\delta}_i$ is a lift of $\delta_i$.
	As for lifts of ordinary paths, such $\hat{\delta}$ exists if and only if $\hat{x}$ is a lift of the startpoint of $\delta_1$, and in this case $\hat{\delta}$ is unique.
\end{defn}

\begin{defn}(Paths along routes)\\
	Let $\fR=(y_0,Y_1,y_1,Y_2,y_2,\dots,Y_n,y_n)$ be a route in $X$.
	A \emph{path along $\fR$} is a segmented path $\delta=(\delta_1,\delta_2,\dots,\delta_n)$, where $\delta_i$ is a path from $y_{i-1}$ to $y_i$ in $X$ that lifts to a path in $Y_i$ from $y_{i-1}^i$ to $y_i^i$.
	If $\fR$ is closed, then we say that $\fR$ is \emph{essential} if every path along $\fR$ is essential.	
\end{defn}

\begin{defn}(Elevations of routes)\\\label{defn:routeelevations}
	Let $\fR=(y_0,Y_1,y_1,Y_2,y_2,\dots,Y_n,y_n)$ be a route in $X$ and let $\hat{X}\to X$ be a cover.
	An \emph{elevation of $\fR$ to $\hat{X}$} is a route $\hat{\fR}=(\hat{y}_0,\hat{Y}_1,\hat{y}_1,\hat{Y}_2,\hat{y}_2,\dots,\hat{Y}_n,\hat{y}_n)$ in $\hat{X}$ such that each $\hat{Y}_i$ is an elevation of $Y_i$ to $\hat{X}$ that gives rise to the following commutative diagram of based spaces:
	\begin{equation}\label{routeelevation}
		\begin{tikzcd}[
			ar symbol/.style = {draw=none,"#1" description,sloped},
			isomorphic/.style = {ar symbol={\cong}},
			equals/.style = {ar symbol={=}},
			subset/.style = {ar symbol={\subset}}
			]
			(\hat{Y}_i,\hat{y}_{i-1}^i,\hat{y}_i^i)\ar{d}\ar{r}&(\hat{X},\hat{y}_{i-1},\hat{y}_i)\ar{d}\\
			(Y_i,y_{i-1}^i,y_i^i)\ar{r}&(X,y_{i-1},y_i)
		\end{tikzcd}
	\end{equation}
\end{defn}

\begin{remk}\label{remk:unirouteelev}
	If $\wt{X}\to X$ is the universal cover of $X$ and $G=\pi_1(X)\acts\wt{X}$ is the action by deck transformations, then for each route $\fR$ in $X$ we get a natural action of $G$ on the set of elevations of $\fR$ to $\wt{X}$ (thus generalising Remark \ref{remk:unibasedelev}).
	More precisely, if $\wt{\fR}=(\ti{y}_0,\wt{Y}_1,\ti{y}_1,\wt{Y}_2,\ti{y}_2,\dots,\wt{Y}_n,\ti{y}_n)$ is an elevation of $\fR$ to $\wt{X}$ and $g\in G$, then $g\wt{\fR}=(g\ti{y}_0,g\wt{Y}_1,g\ti{y}_1,g\wt{Y}_2,g\ti{y}_2,\dots,g\wt{Y}_n,g\ti{y}_n)$ is also an elevation of $\fR$ to $\wt{X}$, and the maps $g\wt{Y}_i\to Y_i$ are just the compositions $g\wt{Y}_i\overset{g^{-1}}{\longrightarrow}\wt{Y}_i\to Y_i$.
\end{remk}

Elevations of routes and lifts of paths along routes are intimately connected, as the following lemma demonstrates.

\begin{lem}\label{lem:routeelev}
	Let $\fR=(y_0,Y_1,y_1,Y_2,y_2,\dots,Y_n,y_n)$ be a route in a non-positively curved cube complex $X$. Let $\hat{X}\to X$ be a cover and $\hat{y}_0$ a lift of $y_0$.
	\begin{enumerate}
		\item\label{item:pathtoelev} For each $\delta$ a path along $\fR$, there is a unique elevation $\hat{\fR}$ of $\fR$ to $\hat{X}$ such that the lift $\hat{\delta}$ of $\delta$ based at $\hat{y}_0$ is a path along $\hat{\fR}$.
		\item\label{item:extendelev} If $1\leq i\leq j\leq n$ and $(\hat{y}_{i-1},\hat{Y}_i,\hat{y}_i,\dots,\hat{Y}_j,\hat{y}_j)$ is an elevation of $(y_{i-1},Y_i,y_i,\dots,Y_j,y_j)$ to $\hat{X}$, then this can be extended to an elevation $(\hat{y}_0,\hat{Y}_1,\hat{y}_1,\hat{Y}_2,\hat{y}_2,\dots,\hat{Y}_n,\hat{y}_n)$ of $\fR$ to $\hat{X}$.
		\item\label{item:elevtopath} For each elevation $\hat{\fR}$ of $\fR$ to $\hat{X}$, each path along $\hat{\fR}$ is a lift of a path along $\fR$.
		\item\label{item:esselev} If $\fR$ is an essential closed route, then every closed elevation of $\fR$ to $\hat{X}$ is essential.
		\item\label{item:essential} Assuming that $\fR$ is closed, $\fR$ is essential if and only if it has no closed elevations to the universal cover $\wt{X}\to X$.
	\end{enumerate}
\end{lem}
\begin{proof}
	\begin{enumerate}
		\item Let $\delta=(\delta_1,\delta_2,\dots,\delta_n)$ be a path along $\fR$.
		Let $\hat{\delta}=(\hat{\delta}_1,\hat{\delta}_2,\dots,\hat{\delta}_n)$ be the lift of $\delta$ to $\hat{X}$ based at $\hat{y}_0$.
		Denote the initial and terminal vertices of $\hat{\delta}_i$ by $\hat{y}_{i-1}$ and $\hat{y}_i$ respectively.
		Let $(\hat{Y}_i,\hat{y}_{i-1}^i)\to(\hat{X},\hat{y}_{i-1})$ be the based elevation of $(Y_i,y^i_{i-1})\to(X,y_{i-1})$ (see Figure \ref{fig:routeelev}).
		As $\delta_i$ lifts to a path in $Y_i$ from $y^i_{i-1}$ to $y^i_i$, it follows from the definition of elevation that $\hat{\delta}_i$ lifts to a path in $\hat{Y}_i$ from $\hat{y}_{i-1}^i$ to some vertex $\hat{y}_i^i$.
		This now provides all the data for an elevation $\hat{\fR}$ of $\fR$ to $\hat{X}$ such that $\hat{\delta}$ is a path along $\hat{\fR}$.
		Moreover, it is clear that this $\hat{\fR}$ is the only elevation of $\fR$ that has $\hat{\delta}$ as a path along it.
		
		\item Similar idea to part \ref{item:pathtoelev}. If $j<n$, let $(\hat{Y}_{j+1},\hat{y}^{j+1}_j)\to(\hat{X},\hat{y}_j)$ be the based elevation of $(Y_{j+1},y^{j+1}_j)\to(X,y_j)$. Choose some lift $\hat{y}^{j+1}_{j+1}\in \hat{Y}_{j+1}$ of $y^{j+1}_{j+1}\in Y_{j+1}$, and let $\hat{y}_{j+1}$ be the image of $\hat{y}^{j+1}_{j+1}$ in $\hat{X}$. Repeat this argument for $Y_{j+2},\dots,Y_n$ to obtain $\hat{Y}_{j+2},\hat{y}_{j+2},\dots,\hat{Y}_n,\hat{y}_n$. Run a similar argument if $i>1$.
		
		\item Let $\hat{\fR}=(\hat{y}_0,\hat{Y}_1,\hat{y}_1,\hat{Y}_2,\hat{y}_2,\dots,\hat{Y}_n,\hat{y}_n)$ be an elevation of $\fR$ to $\hat{X}$, and let $\hat{\delta}=(\hat{\delta}_1,\hat{\delta}_2,\dots,\hat{\delta}_n)$ be a path along $\hat{\fR}$.
		Say $\hat{\delta}$ descends to a segmented path $\delta=(\delta_1,\delta_2,\dots,\delta_n)$ in $X$.
		Since each $\hat{\delta}_i$ is a path from $\hat{y}_{i-1}$ to $\hat{y}_i$ in $\hat{X}$ that lifts to a path in $\hat{Y}_i$ from $\hat{y}_{i-1}^i$ to $\hat{y}_i^i$, it follows from diagram (\ref{routeelevation}) that each $\delta_i$ is a path from $y_{i-1}$ to $y_i$ in $X$ that lifts to a path in $Y_i$ from $y_{i-1}^i$ to $y_i^i$.
		Thus $\delta$ is a path along $\fR$.
		
		\item Let $\hat{\fR}$ be an elevation of $\fR$ to $\hat{X}$.
		Suppose for contradiction that $\hat{\fR}$ is not essential.
		Then there exists a null-homotopic path $\hat{\delta}$ along $\hat{\fR}$.
		By \ref{item:elevtopath}, $\hat{\delta}$ descends to a path along $\fR$, which will also be null-homotopic, contradicting essentialness of $\fR$.
		
		\item $\fR$ is essential if and only if every path along $\fR$ is essential, which is equivalent to no path along $\fR$ having a closed lift to the universal cover $\wt{X}\to X$.
		By \ref{item:pathtoelev} and \ref{item:elevtopath}, this is also equivalent to $\fR$ having no closed elevations to $\wt{X}$.\qedhere
	\end{enumerate}	
\end{proof}

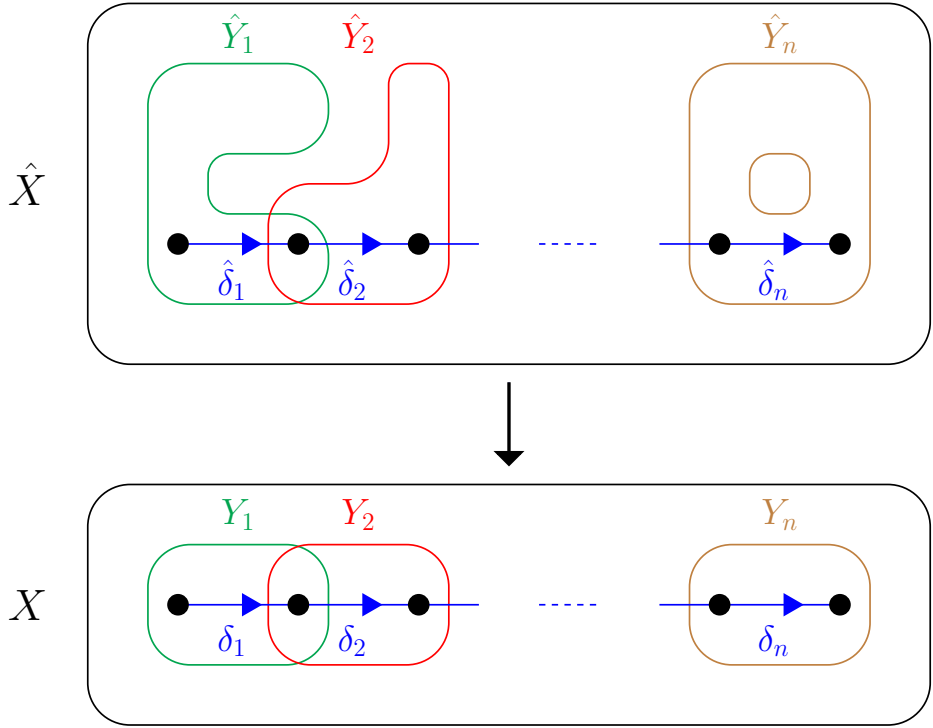
\begin{figure}[H]
	\centering
	\scalebox{0.8}{
		\begin{tikzpicture}[auto,node distance=2cm,
			thick,every node/.style={circle,draw,font=\small},
			every loop/.style={min distance=2cm},
			hull/.style={draw=none},
			]
			\tikzstyle{label}=[draw=none,font=\LARGE]
			
			\draw[rounded corners=20pt] (0,6) rectangle (14,12) {};
			\draw[rounded corners=20pt,Green] (3,8.5)--(4,8.5)--(4,7)--(1,7) --(1,11)
			--(4,11)--(4,9.5)[rounded corners=10pt]--(2,9.5)
			--(2,8.5)--(3,8.5) {};
			\draw[rounded corners=20pt,red](4,7)--(3,7)--(3,9)--(5,9)
			[rounded corners=10pt]--(5,11)--(6,11)[rounded corners=20pt]--(6,7)--(4,7){};
			\draw[rounded corners=20pt,brown](10,7) rectangle (13,11);
			\draw[rounded corners=10pt,brown](11,8.5) rectangle (12,9.5);
			
			\path (1.5,8) edge [blue,postaction={decoration={markings,mark=at position 0.7 with {\arrow[blue,line width=1mm, scale=0.7]{triangle 60}}},decorate}] (3.5,8);
			\path (3.5,8) edge [blue,postaction={decoration={markings,mark=at position 0.7 with {\arrow[blue,line width=1mm, scale=0.7]{triangle 60}}},decorate}] (5.5,8);
			\path (5.5,8) edge [blue] (6.5,8);
			\path (7.5,8) edge [blue,dashed] (8.5,8);
			\path (9.5,8) edge [blue] (10.5,8);			
			\path (10.5,8) edge [blue,postaction={decoration={markings,mark=at position 0.7 with {\arrow[blue,line width=1mm, scale=0.7]{triangle 60}}},decorate}] (12.5,8);
			\node[fill] at (1.5,8){};
			\node[fill] at (3.5,8){};
			\node[fill] at (5.5,8){};
			\node[fill] at (10.5,8){};
			\node[fill] at (12.5,8){};
			
			\node[label,font=\huge] (Xh) at (-1,9) {$\hat{X}$};
			\node[label,Green] (Y1h) at (2.5,11.5) {$\hat{Y}_1$};
			\node[label,blue] (d1h) at (2.4,7.4) {$\hat{\delta}_1$};
			\node[label,red] (Y2h) at (4.5,11.5) {$\hat{Y}_2$};
			\node[label,blue] (d2h) at (4.4,7.4) {$\hat{\delta}_2$};
			\node[label,brown] (Ynh) at (11.5,11.5) {$\hat{Y}_n$};
			\node[label,blue] (dnh) at (11.4,7.4) {$\hat{\delta}_n$};
			
			\draw[draw=black,fill=blue,-triangle 90, ultra thick] (7,5.7) -- (7,4.3);
			
			\draw[rounded corners=20pt] (0,0) rectangle (14,4) {};
			\draw[rounded corners=20pt,Green] (1,1) rectangle (4,3) {};
			\draw[rounded corners=20pt,red] (3,1) rectangle (6,3) {};
			\draw[rounded corners=20pt,brown] (10,1) rectangle (13,3) {};
			\path (1.5,2) edge [blue,postaction={decoration={markings,mark=at position 0.7 with {\arrow[blue,line width=1mm, scale=0.7]{triangle 60}}},decorate}] (3.5,2);
			\path (3.5,2) edge [blue,postaction={decoration={markings,mark=at position 0.7 with {\arrow[blue,line width=1mm, scale=0.7]{triangle 60}}},decorate}] (5.5,2);
			\path (5.5,2) edge [blue] (6.5,2);
			\path (7.5,2) edge [blue,dashed] (8.5,2);
			\path (9.5,2) edge [blue] (10.5,2);			
			\path (10.5,2) edge [blue,postaction={decoration={markings,mark=at position 0.7 with {\arrow[blue,line width=1mm, scale=0.7]{triangle 60}}},decorate}] (12.5,2);
			\node[fill] at (1.5,2){};
			\node[fill] at (3.5,2){};
			\node[fill] at (5.5,2){};
			\node[fill] at (10.5,2){};
			\node[fill] at (12.5,2){};
			
			\node[label,font=\huge] (X) at (-1,2) {$X$};
			\node[label,Green] (Y1) at (2.5,3.5) {$Y_1$};
			\node[label,blue] (d1) at (2.4,1.4) {$\delta_1$};
			\node[label,red] (Y2) at (4.5,3.5) {$Y_2$};
			\node[label,blue] (d2) at (4.4,1.4) {$\delta_2$};
			\node[label,brown] (Yn) at (11.5,3.5) {$Y_n$};
			\node[label,blue] (dn) at (11.4,1.4) {$\delta_n$};

		\end{tikzpicture}
	}
	\caption{Picture for Lemma \ref{lem:routeelev}\ref{item:pathtoelev} (in the case where the routes are embedded).}\label{fig:routeelev}
\end{figure}

\begin{remk}
	With the assumptions of Lemma \ref{lem:routeelev}, it follows from Lemma \ref{lem:routeelev}\ref{item:pathtoelev} that there exists an elevation $\hat{\fR}$ of $\fR$ to $\hat{X}$ with initial vertex $\hat{y}_0$. However, such an elevation might not be unique because different segmented paths $\delta$ might give rise to different elevations $\hat{\fR}$.
\end{remk}

The following lemma lists some more basic properties about elevations of routes. The proofs are exercises in elementary covering space theory.

\begin{lem}\label{lem:elevelev}
	Let $X''\to X'\to X$ be covering maps between non-positively curved cube complexes.
	Let $\fR$ be a route in $X$.
	\begin{enumerate}
		\item If $\fR$ is embedded, then any elevation of $\fR$ to $X'$ is embedded.
		\item\label{item:closedtoclosed} If $\fR'$ is a closed elevation of $\fR$ to $X'$, then $\fR$ is closed.
		\item If $\fR'$ is an elevation of $\fR$ to $X'$, and $\fR''$ is an elevation of $\fR'$ to $X''$, then $\fR''$ is an elevation of $\fR$ to $X''$.
		\item\label{item:R''toR'} If $\fR''$ is an elevation of $\fR$ to $X''$, then there exists a unique elevation $\fR'$ of $\fR$ to $X'$ such that $\fR''$ is an elevation of $\fR'$.
		\item\label{item:noclosedelev} If $\fR$ has no closed elevations to $X'$, then $\fR$ also has no closed elevations to $X''$.
	\end{enumerate}
\end{lem}

When constructing a finite cover with no closed elevations of a given route (as we will do later in Theorems \ref{thm:embroutes} and \ref{thm:routesext}), it is often helpful to do the construction in stages, first working with a certain intermediate cover.
The following lemma will be very useful for this.

\begin{lem}\label{lem:factorarg}
	Let $X'\to X$ be a finite-sheeted covering map of non-positively curved cube complexes, and let $\fR$ be a finite closed route in $X$. Suppose that, for each closed elevation $\fR'$ of $\fR$ to $X'$, there is a finite-sheeted cover $\hat{X}_{\fR'}\to X'$ with no closed elevations of $\fR'$.
	Then there is a finite-sheeted cover $X''\to X$ with no closed elevations of $\fR$.
\end{lem}
\begin{proof}
	There are finitely many closed elevations $\fR'$ of $\fR$ to $X'$, so finitely many covers $\hat{X}_{\fR'}\to X'$. Hence, there is a finite-sheeted cover $X''\to X$ that factors through each of the $\hat{X}_{\fR'}$.
	Suppose for contradiction that there exists a closed elevation $\fR''$ of $\fR$ to $X''$.
	By Lemma \ref{lem:elevelev}\ref{item:R''toR'}, there exists a (unique) elevation $\fR'$ of $\fR$ to $X'$ such that $\fR''$ is an elevation of $\fR'$.
	Applying Lemma \ref{lem:elevelev}\ref{item:R''toR'} again, there is an elevation $\hat{\fR}$ of $\fR'$ to $\hat{X}_{\fR'}$ such that $\fR''$ is an elevation of $\hat{\fR}$. By Lemma \ref{lem:elevelev}\ref{item:closedtoclosed}, $\hat{\fR}$ is closed, contradicting the hypothesis on $\hat{X}_{\fR'}$.
\end{proof}

Here is another useful lemma, which allows us to understand the set of all possible elevations of a given route to the universal cover.

\begin{lem}\label{lem:elevuniv}
	Let $X$ be a non-positively curved cube complex with fundamental group $G$ and universal cover $\wt{X}$.
	Let $\wt{\fR}=(\ti{y}_0,\wt{Y}_1,\ti{y}_1,\wt{Y}_2,\ti{y}_2,\dots,\wt{Y}_n,\ti{y}_n)$ be a route in $\wt{X}$, and let $K_i<G_{\wt{Y}_i}$ for $1\leq i\leq n$.
	Consider the route
	\begin{equation*}
		\fR=(G\cdot\ti{y}_0,\wt{Y}_1/K_1,G\cdot\ti{y}_1,\wt{Y}_2/K_2,\dots,\wt{Y}_n/K_n,G\cdot\ti{y}_n)
	\end{equation*}
	in $X$, where the maps $\wt{Y}_i/K_i\to X=\wt{X}/G$ are the natural quotient maps, and each complex $\wt{Y}_i/K_i$ is equipped with vertices $K_i\cdot\ti{y}_{i-1},K_i\cdot\ti{y}_i$ (which clearly map to $G\cdot\ti{y}_{i-1},G\cdot\ti{y}_i$ in $X$).
	Then every elevation of $\fR$ to $\wt{X}$ takes the form
	\begin{equation}\label{wtfR'}
		\wt{\fR}'=(g_0\ti{y}_0,g_0\wt{Y}_1,g_1\ti{y}_1,g_1\wt{Y}_2,g_2\ti{y}_2,\dots,g_{n-1}\wt{Y}_n,g_n\ti{y}_n),
	\end{equation}
	where $g_0\in G$, and $g_i=g_0k_1k_2\cdots k_i$ ($1\leq i\leq n$), with $k_j\in K_j$.	
	Conversely, every route of the form (\ref{wtfR'}) is an elevation of $\fR$.
	Moreover, each covering map $g_{i-1}\wt{Y}_i\to\wt{Y}_i/K_i$ is the composition of left multiplication by $g_{i-1}^{-1}$ with the natural quotient map $\wt{Y}_i\to\wt{Y}_i/K_i$.
\end{lem}
\begin{remk}
	Note that the $\wt{Y}_i$ are all embedded in $\wt{X}$ (Remark \ref{remk:isembedding}), so we can consider them as convex subcomplexes of $\wt{X}$. Hence the stabiliser notation $G_{\wt{Y}_i}$ and the translate notation $g_{i-1}\wt{Y}_i$ makes sense.
\end{remk}
\begin{proof}[Proof of Lemma \ref{lem:elevuniv}]
	Let $\wt{\fR}'=(\ti{y}'_0,\wt{Y}'_1,\ti{y}'_1,\wt{Y}'_2,\ti{y}'_2,\dots,\wt{Y}'_n,\ti{y}'_n)$ be an elevation of $\fR$ to $\wt{X}$.
	We claim that $\wt{\fR}'$ is of the form (\ref{wtfR'}).
	For $1\leq i\leq n$, it is clear that $(\wt{Y}_i,\ti{y}_{i-1})\to(\wt{X},\ti{y}_{i-1})$ is one based elevation of $(\wt{Y}_i/K_i,K_i\cdot\ti{y}_{i-1})\to(X,G\cdot\ti{y}_{i-1})$ to $\wt{X}$.
	Remark \ref{remk:unibasedelev} says that $G$ acts transitively on the set of based elevations, so we must have $(\wt{Y}'_i,\ti{y}'_{i-1})=(g_{i-1}\wt{Y}_i,g_{i-1}\ti{y}_{i-1})$ for some $g_{i-1}\in G$.
	Plus this tells us that the covering map $\wt{Y}'_i\to\wt{Y}_i/K_i$ is the composition of left multiplication by $g_{i-1}^{-1}$ with the natural quotient map $\wt{Y}_i\to\wt{Y}_i/K_i$.
	We also have $\ti{y}'_n=g_n\ti{y}_n$ for some $g_n\in G$, since $\ti{y}'_n$ and $\ti{y}_n$ must have the same image in $X$.
	Since $\wt{\fR}'$ is an elevation of $\fR$, we know that $\ti{y}'_i=g_i\ti{y}_i$ maps to $K_i\cdot \ti{y}_i$ under the covering map $\wt{Y}'_i\to\wt{Y}_i/K_i$ (again $1\leq i\leq n$), so $K_i\cdot g_{i-1}^{-1}g_i\ti{y}_i=K_i\cdot\ti{y}_i$. As $K_i$ acts freely on $\wt{X}$, we see that $g_{i-1}^{-1}g_i\in K_i$.
	Therefore, $\wt{\fR}'$ is of the form (\ref{wtfR'}), as claimed.
	
	Conversely, let $\wt{\fR}'$ take the form (\ref{wtfR'}). Again by Remark \ref{remk:unibasedelev}, each complex $g_{i-1}\wt{Y}_i$ is an elevation of $\wt{Y}_i/K_i\to X$, with covering map  $g_{i-1}\wt{Y}_i\to\wt{Y}_i/K_i$ defined as the composition of left multiplication by $g_{i-1}^{-1}$ with the natural quotient map $\wt{Y}_i\to\wt{Y}_i/K_i$.
	As $g_{i-1}^{-1}g_i=k_i\in K_i$, this map sends the vertices $g_{i-1}\ti{y}_{i-1},g_i\ti{y}_i\in g_{i-1}\wt{Y}_i$ to the vertices $K_i\cdot\ti{y}_{i-1},K_i\cdot\ti{y}_i\in\wt{Y}_i/K_i$.
	Thus $\wt{\fR}'$ is an elevation of $\fR$.
\end{proof}

We finish this subsection with a proposition about routes in virtually special cube complexes.

\begin{prop}\label{prop:routeembed}
	Let $\fR=(y_0,Y_1,y_1,Y_2,y_2,\dots,Y_n,y_n)$ be a finite route in a finite virtually special cube complex $X$.
	Then there is a finite cover $\hat{X}\to X$ such that every elevation $\hat{\fR}$ of $\fR$ to $\hat{X}$ is embedded, and such that the subcomplexes $\hat{Y}_i$ in $\hat{\fR}$ do not inter-osculate with hyperplanes of $\hat{X}$.
\end{prop}
\begin{proof}
	Apply Proposition \ref{prop:elevembed} to the local isometries $Y_1,\dots,Y_n\to X$.
\end{proof}

\subsection{Connection between routes and product separability}

The separability of a subgroup can be characterized geometrically using finite-sheeted covers and elevations of complexes \cite{Scott78} (see also \cite[Lemma 4.8]{WiseRiches}).
In the following proposition, we use routes to provide a similar geometric criterion for the separability of products of convex-cocompact subgroups in the fundamental group of a  non-positively curved cube complex.

\begin{prop}\label{prop:routesep}
	Let $X$ be a non-positively curved cube complex which is either locally finite or virtually special, and let $n\geq1$ be an integer. The following are equivalent:
	\begin{enumerate}
		\item\label{item:routes} For any finite essential closed route $\fR$ in $X$ of length $n$, there exists a finite-sheeted cover $\hat{X}\to X$ with no closed elevations of $\fR$.
		\item\label{item:prodsep} Any product of $n$ convex-cocompact subgroups in $G=\pi_1(X)$ is separable.
	\end{enumerate}
\end{prop}

\begin{remk}
	The ideas behind Proposition \ref{prop:routesep} are quite general; indeed one could make an analogue of the proposition for $X$ a cell complex (in this case the convex-cocompact subgroups would just be finitely presented subgroups, and the routes would be built from maps of finite cell complexes rather than cube complexes).
\end{remk}

\begin{remk}
	The assumption that $X$ is either locally finite or virtually special in Proposition \ref{prop:routesep} will only be used in the proof of the implication \ref{item:routes} $\implies$ \ref{item:prodsep}.
\end{remk}

Before proving Proposition \ref{prop:routesep}, we establish the following lemma.
Two pieces of notation for this lemma: if $\gamma$ is a path in a cube complex, then $[\gamma]$ denotes the homotopy class of $\gamma$ rel endpoints, and $\bar{\gamma}$ denotes the path obtained from $\gamma$ by reversing the orientation.

\begin{lem}\label{lem:realizations}
	Let $\fR=(y_0,Y_1,y_1,Y_2,y_2,\dots,Y_n,y_n)$ be a route in a non-positively curved cube complex $X$. Let $\gamma=(\gamma_1,\gamma_2,\dots,\gamma_n)$ be a path along $\fR$.
	Let $K_i<G=\pi_1(X,y_0)$ be the subgroup consisting of elements $[\gamma_1\dots\gamma_{i-1}\eta_i\bar{\gamma}_{i-1}\dots\bar{\gamma}_1]$, where $\eta_i$ is a loop in $X$ based at $y_{i-1}$ that lifts to a loop in $Y_i$ based at $y_{i-1}^i$.
	Then $K_1K_2\cdots K_n[\gamma_1\gamma_2\cdots\gamma_n]$ is precisely the set of homotopy classes of realizations of paths along $\fR$.
\end{lem}
\begin{proof}
	Pick elements $k_i=[\gamma_1\dots\gamma_{i-1}\eta_i\bar{\gamma}_{i-1}\dots\bar{\gamma}_1]\in K_i$ for each $i$.
	Then $k_1k_2\cdots k_n[\gamma_1\gamma_2\cdots\gamma_n]=[\eta_1\gamma_1\eta_2\gamma_2\cdots \eta_n\gamma_n]$. This is the homotopy class of the realization of $(\eta_1\gamma_1,\eta_2\gamma_2,\dots,\eta_n\gamma_n)$, which is a path along $\fR$, and it is not hard to see that every path along $\fR$ is of this form, up to homotopy of the segments.
\end{proof}

\begin{proof}[Proof of Proposition \ref{prop:routesep}]
	\ref{item:routes} $\implies$ \ref{item:prodsep}:
	Let $K_1,\dots,K_n<G$ be convex-cocompact subgroups.
	Let $\wt{X}\to X$ be the universal cover, $x\in X$ a vertex and $\ti{x}$ a lift to $\wt{X}$.
	We will work with $G=\pi_1(X,x)$, and the action on $\wt{X}$ will be induced by $x$ and $\ti{x}$.	
	Let $g\in G - K_1K_2\cdots K_n$.	
	Each $K_i$ stabilizes a convex subcomplex $\wt{Y}_i\subset\wt{X}$ with finite quotient $\wt{Y}_i/K_i$.
	Enlarging each $\wt{Y}_i$ if necessary, we can assume that they all contain $\ti{x}$ and $g^{-1}\ti{x}$; more precisely, if $X$ (and hence $\wt{X}$) is locally finite, then we may replace each $\wt{Y}_i$ with a cubical thickening, and use the fact that the action of $K_i$ on any cubical thickening of $\wt{Y}_i$ is still cocompact; on the other hand, if $X$ is virtually special then we construct the enlargements of the $\wt{Y}_i$ using Lemma \ref{lem:vspeccocomp}. 
	
	Now consider the finite route
	$$\fR=(x, \wt{Y}_1/K_1,x,\wt{Y}_2/K_2,x,\dots,\wt{Y}_n/K_n,x)$$
	in $X$. The maps $\wt{Y}_i/K_i\to X=\wt{X}/G$ are just quotient maps, and for the basepoints in $\wt{Y}_i/K_i$ we take $K_i\cdot\ti{x}$ -- except for $\wt{Y}_n/K_n$, where we take  $K_n\cdot \ti{x}$ and $K_n\cdot g^{-1}\ti{x}$ (in that order).
	Let $\gamma_1,\dots,\gamma_{n-1}$ all be the trivial path based at $x$ and let $\gamma_n$ be a loop based at $x$ that lifts to a path in $\wt{Y}_n$ from $\ti{x}$ to $g^{-1}\ti{x}$.
	Then $\gamma=(\gamma_1,\gamma_2,\dots,\gamma_n)$ is a path along $\fR$.
	By construction, we also have $[\gamma_1\gamma_2\cdots\gamma_n]=g^{-1}$.	
	
	Lemma \ref{lem:realizations} tells us that $K_1K_2\cdots K_n[\gamma_1\gamma_2\cdots\gamma_n]=K_1K_2\cdots K_ng^{-1}$ is precisely the set of homotopy classes of realizations of paths along $\fR$.
	(The way we have set things up means that the subgroups $K_i$ from Lemma \ref{lem:realizations} are precisely the subgroups $K_i$ we started the proof with.)
	Since $g\notin K_1K_2\cdots K_n$, the route $\fR$ is essential.
	
	By our assumption \ref{item:routes}, there is a finite-sheeted cover $\hat{X}\to X$ with no closed elevations of $\fR$. Let $\hat{G}<G$ be the subgroup corresponding to $\hat{X}$.
	By Lemma \ref{lem:routeelev}\ref{item:pathtoelev}, no path along $\fR$ (or realization of a path along $\fR$) has a closed lift to $\hat{X}$, so we deduce that $1\notin \hat{G}K_1K_2\cdots K_ng^{-1}$, or equivalently $g\notin \hat{G}K_1K_2\cdots K_n$.
	Since $\hat{G}$ has finite index in $G$, it follows that $K_1K_2\cdots K_n$ is separable in $G$, as required (see Lemma \ref{lem:sepequiv}).

	\ref{item:prodsep} $\implies$ \ref{item:routes}: Let $\fR=(y_0,Y_1,y_1,Y_2,y_2,\dots,Y_n,y_n)$ be a finite essential closed route in $X$.
Fix a path $\gamma=(\gamma_1,\gamma_2,\dots,\gamma_n)$ along $\fR$.
Lemma \ref{lem:realizations} then provides us with convex-cocompact subgroups $K_1,K_2,\dots,K_n<G$ such that $K_1K_2\cdots K_n[\gamma_1\gamma_2\cdots\gamma_n]$ is precisely the set of homotopy classes of realizations of paths along $\fR$.

The route $\fR$ is essential, so $1\notin K_1K_2\cdots K_n[\gamma_1\gamma_2\cdots\gamma_n]$.
The product $K_1K_2\cdots K_n$ is separable by \ref{item:prodsep}, hence there is a finite-index normal subgroup $\hat{G}\triangleleft G$ with $1\notin\hat{G}K_1K_2\cdots K_n[\gamma_1\gamma_2\cdots\gamma_n]$ (Lemma \ref{lem:sepequiv}).
Again using the fact that $K_1K_2\cdots K_n[\gamma_1\gamma_2\cdots\gamma_n]$ is precisely the set of homotopy classes of realizations of paths along $\fR$, we deduce that no path along $\fR$ has a closed lift to the finite cover $\hat{X}\to X$ corresponding to $\hat{G}$.
Finally, we deduce from Lemma \ref{lem:routeelev}\ref{item:elevtopath} that $\hat{X}$ has no closed elevation of $\fR$.
\end{proof}

\bigskip
\section{Controlling routes and hyperplanes}\label{sec:control}

In this section and the next we will prove the following theorem.

\begin{thm}\label{thm:embroutes}
	Let $X$ be a finite directly special cube complex and let $n\geq1$ be an integer.
	For any essential embedded closed route $\fR$ in $X$ of length $n$, there exists a finite cover $\hat{X}\to X$ with no closed elevations of $\fR$.
\end{thm}

We will then use Theorem \ref{thm:embroutes} to prove Theorem \ref{thm:prodsep} in Section \ref{sec:proof}.
The proof of Theorem \ref{thm:embroutes} will be an induction on the length $n$ of the route, with base case $n\leq3$ following from the separability of products of 3 convex-cocompact subgroups in $\pi_1(X)$ \cite{Shepherd23} (together with Proposition \ref{prop:routesep}).
The inductive argument will be built from a number of lemmas, which will be combined together in Section \ref{sec:walker}.

\subsection{Separating $Y_2$ and $Y_n$}

\begin{lem}\label{lem:2intn}
Let $\fR=(y_0,Y_1,y_1,Y_2,y_2,\dots,Y_n,y_n)$ be an essential embedded closed route in a finite non-positively curved cube complex $X$, with $n\geq4$.
If Theorem \ref{thm:embroutes} holds for routes of length less than $n$, then there is a finite cover $\hat{X}\to X$ such that every closed elevation $\hat{\fR}=(\hat{y}_0,\hat{Y}_1,\hat{y}_1,\hat{Y}_2,\hat{y}_2,\dots,\hat{Y}_n,\hat{y}_n)$ of $\fR$ to $\hat{X}$ has $\hat{Y}_2\cap\hat{Y}_n=\emptyset$.
\end{lem}
\begin{proof}
If we already have $Y_2\cap Y_n=\emptyset$ then we can just take $\hat{X}=X$, so assume $Y_2\cap Y_n\neq\emptyset$.
Let $x\in Y_2\cap Y_n$ be a vertex, and consider the following closed routes:
$$\fR_x^1=(y_0,Y_1,y_1,Y_2,x,Y_n,y_n)\quad\text{and}\quad\fR_x^2=(x,Y_2,y_2,Y_3,y_3,\dots,y_{n-1},Y_n,x)$$

We claim that one of $\fR_x^1$ and $\fR_x^2$ is essential. Indeed, otherwise there exists a null-homotopic path $(\gamma_1,\gamma_2,\gamma_n)$ along $\fR_x^1$ and a null-homotopic path $(\delta_2,\delta_3,\dots,\delta_n)$ along $\fR_x^2$ (we choose the indices here to match the indices of the $Y_i$).
But then $(\gamma_1,\gamma_2\delta_2,\delta_3,\dots,\delta_{n-1},\delta_n\gamma_n)$ is a null-homotopic path along $\fR$, contradicting the essentialness of $\fR$.

By hypothesis, Theorem \ref{thm:embroutes} holds for routes of length less than $n$.
So if $\fR_x^1$ (resp. $\fR_x^2$) is essential, then there exists a finite cover $\hat{X}_x\to X$ with no closed elevations of $\fR_x^1$ (resp. $\fR_x^2$).
(If $\fR_x^1$ and $\fR_x^2$ are both essential then just pick one of them when defining $\hat{X}_x$.)
Letting $\hat{X}\to X$ be a finite cover that factors through all of the finitely many $\hat{X}_x\to X$ (where $x$ ranges over all vertices in $Y_2\cap Y_n$), we see that there is no $x$ such that $\fR_x^1$ and $\fR_x^2$ both have a closed elevation to $\hat{X}$ (using Lemma \ref{lem:elevelev}\ref{item:noclosedelev}). 

We claim that each closed elevation $\hat{\fR}=(\hat{y}_0,\hat{Y}_1,\hat{y}_1,\hat{Y}_2,\hat{y}_2,\dots,\hat{Y}_n,\hat{y}_n)$  of $\fR$ to $\hat{X}$ satisfies $\hat{Y}_2\cap\hat{Y}_n=\emptyset$.
Indeed, suppose for contradiction that $\hat{Y}_2\cap\hat{Y}_n\neq\emptyset$.
Say $\hat{x}\in \hat{Y}_2\cap\hat{Y}_n$ is a vertex, and let $x$ be the image of $\hat{x}$ in $X$ (note that $x\in Y_2\cap Y_n$).
But then the routes
$$\hat{\fR}_x^1=(\hat{y}_0,\hat{Y}_1,\hat{y}_1,\hat{Y}_2,\hat{x},\hat{Y}_n,\hat{y}_n)\quad\text{and}\quad\hat{\fR}_x^2=(\hat{x},\hat{Y}_2,\hat{y}_2,\hat{Y}_3,\hat{y}_3,\dots,\hat{y}_{n-1},\hat{Y}_n,\hat{x})$$
are closed elevations of $\fR_x^1$ and $\fR_x^2$ respectively to $\hat{X}$, contradicting the construction of $\hat{X}$.
\end{proof}

\subsection{The complexity $\kappa_j(\fR)$}\label{subsec:kappaj}

The arguments in Section \ref{subsec:Hypj} will require induction not just on the length $n$ of the route, but on an additional measure of complexity $\kappa_j(\fR)$, which we define in this section.
First, we define a certain collection of subcomplexes $\cP_j(\fR)$ associated to a route $\fR$.

\begin{defn}(The collection of subcomplexes $\cP_j=\cP_j(\fR)$)\\\label{defn:cP}
Let $\fR=(y_0,Y_1,y_1,Y_2,y_2,\dots,Y_n,y_n)$ be an embedded route in a finite directly special cube complex $X$, with $n\geq4$. For $2\leq j\leq n-1$, let $\cP_j=\cP_j(\fR)$ denote the collection of locally convex subcomplexes $Z\subset Y_1$ such that there exists an elevation $\wt{\fR}=(\ti{y}_0,\wt{Y}_1,\ti{y}_1,\wt{Y}_2,\ti{y}_2,\dots,\wt{Y}_n,\ti{y}_n)$ of $\fR$ to the universal cover $\wt{X}\to X$ with $\Pi_{\wt{Y}_1}(\wt{Y}_j)$ an elevation of $Z$ (remember that $\Pi_{\wt{A}}$ denotes the projection map $\wt{X}\to\wt{A}$ -- see Definition \ref{defn:projection}).
\end{defn}

\begin{remk}\label{remk:iselevation}
	The subcomplex $\Pi_{\wt{Y}_1}(\wt{Y}_j)\subset\wt{X}$ is always an elevation of some locally convex subcomplex of $Y_1$ by Proposition \ref{prop:projection}, so each $\cP_j$ is non-empty.
\end{remk}

\begin{remk}\label{remk:cP2}
	For any elevation $\wt{\fR}$ of $\fR$, we have $\Pi_{\wt{Y}_1}(\wt{Y}_2)=\wt{Y}_1\cap\wt{Y}_2$, and this is an elevation of the component of $Y_1\cap Y_2$ that contains $y_1$. Therefore, $\cP_2$ consists of just a single subcomplex.
\end{remk}

\begin{defn}(The quasiorder on $\cP_j$)\\\label{defn:quasiorder}
	Let $\cP_j$ be as in Definition \ref{defn:cP}.
	Let $\leq$ be the quasiorder on $\cP_j$ defined by $Z_1\leq Z_2$ if there are elevations $\wt{Z}_1,\wt{Z}_2$ to $\wt{X}$ of $Z_1,Z_2$ respectively with $\hH(\wt{Z}_1)\subset\hH(\wt{Z}_2)$.
	(Reflexivity of $\leq$ is immediate; transitivity follows easily from the observation that, for each $Z\in\cP_j$, the deck group of $\wt{X}\to X$ acts transitively on the set of elevations of $Z$ -- see Remark \ref{remk:univelev}.)
\end{defn}

\begin{defn}(The complexity $\kappa_j=\kappa_j(\fR)$)\\\label{defn:kappaj}
	Define an equivalence relation $\sim$ on $\cP_j$ where $Z_1\sim Z_2$ if $Z_1\leq Z_2\leq Z_1$.
	Denote each equivalence class by $[Z_1]_j$.
	The quasiorder $\leq$ on $\cP_j$ induces a partial order on $\cP_j/\sim$, which we again denote by $\leq$.
	Write $\kappa_j=\kappa_j(\fR)$ for the length of the longest $\leq$-chain in $\cP_j/\sim$.
	Note that $\kappa_j$ is finite because $\cP_j$ (and hence $\cP_j/\sim$) is finite. 
	We think of the numbers $\kappa_j(\fR)$ as measures of complexity for $\fR$.
	Since $\kappa_j$ is finite, $\cP_j/\sim$ contains a $\leq$-maximal element.	
\end{defn}

\begin{lem}\label{lem:Z1Z2Z1}
	If $Z_1\leq Z_2\leq Z_1$ in $\cP_j$, and if $\wt{Z}_1,\wt{Z}_2$ are elevations to $\wt{X}$ of $Z_1,Z_2$ respectively with $\hH(\wt{Z}_1)\subset\hH(\wt{Z}_2)$, then $\hH(\wt{Z}_1)=\hH(\wt{Z}_2)$.
\end{lem}
\begin{proof}
	As $Z_2\leq Z_1$, there are elevations $\wt{Z}'_2, \wt{Z}'_1$ to $\wt{X}$ with $\hH(\wt{Z}'_2)\subset\hH(\wt{Z}'_1)$.
	Applying an appropriate deck transformation to $\wt{Z}'_2, \wt{Z}'_1$, we may assume that $\wt{Z}_2=\wt{Z}'_2$.
	Then $\hH(\wt{Z}_1)\subset\hH(\wt{Z}_2)\subset\hH(\wt{Z}'_1)$, so Lemma \ref{lem:hH=hH} implies that $\hH(\wt{Z}_1)=\hH(\wt{Z}_2)=\hH(\wt{Z}'_1)$.
\end{proof}

\begin{lem}\label{lem:leqmaximal}
$[Z_1]_j\in\cP_j/\sim$ is $\leq$-maximal if and only if there is no $Z_2\in\cP_j$ and elevations $\wt{Z}_1,\wt{Z}_2$ to $\wt{X}$ with $\hH(\wt{Z}_1)\subsetneq\hH(\wt{Z}_2)$.
\end{lem}
\begin{proof}
	If $[Z_1]_j$ is not $\leq$-maximal then there is $[Z_2]_j\in\cP_j/\sim$ with $[Z_1]_j\leq[Z_2]_j$ and $[Z_1]_j\neq[Z_2]_j$. 
	Hence $Z_1\leq Z_2$ but $Z_2\nleq Z_1$. This implies that there are elevations $\wt{Z}_1,\wt{Z}_2$ of $Z_1,Z_2$ to $\wt{X}$ with $\hH(\wt{Z}_1)\subsetneq\hH(\wt{Z}_2)$.
	
	Conversely, suppose there is $Z_2\in\cP_j$ and elevations $\wt{Z}_1,\wt{Z}_2$ of $Z_1,Z_2$ to $\wt{X}$ with $\hH(\wt{Z}_1)\subsetneq\hH(\wt{Z}_2)$. Then $Z_1\leq Z_2$, and, by Lemma \ref{lem:Z1Z2Z1}, $Z_2\nleq Z_1$. Hence $[Z_1]_j$ is not $\leq$-maximal.
\end{proof}

The following lemma shows that the complexity $\kappa_j$ can only decrease when passing to a finite cover.

\begin{lem}\label{lem:kappadown}
	Let $\fR=(y_0,Y_1,y_1,Y_2,y_2,\dots,Y_n,y_n)$ be an embedded route in a finite directly special cube complex $X$, with $n\geq4$. Let $\hat{\fR}$ be an elevation of $\fR$ to a finite cover $\hat{X}\to X$ and let $2\leq j\leq n-1$.
	Then each complex in $\cP_j(\hat{\fR})$ is the elevation of a complex in $\cP_j(\fR)$. Moreover, this induces a $\leq$-preserving map $\cP_j(\hat{\fR})/\sim\to\cP_j(\fR)/\sim$, which sends each $\leq$-chain in $\cP_j(\hat{\fR})/\sim$ to a $\leq$-chain in $\cP_j(\fR)/\sim$ of the same length.
	In particular, $\kappa_j(\hat{\fR})\leq\kappa_j(\fR)$.
\end{lem}
\begin{proof}
	Let $\hat{Z}\in\cP_j(\hat{\fR})$. Then there is an elevation $\wt{\fR}=(\ti{y}_0,\wt{Y}_1,\ti{y}_1,\wt{Y}_2,\ti{y}_2,\dots,\wt{Y}_n,\ti{y}_n)$ of $\hat{\fR}$ to the universal cover $\wt{X}\to X$ with $\Pi_{\wt{Y}_1}(\wt{Y}_j)$ an elevation of $\hat{Z}$.
	But $\wt{\fR}$ is also an elevation of $\fR$, so Remark \ref{remk:iselevation} implies that $\Pi_{\wt{Y}_1}(\wt{Y}_j)$ and $\hat{Z}$ are both elevations of some complex $Z\in\cP_j(\fR)$.
	
	If $\hat{Z}_1\leq\hat{Z}_2$ in $\cP_j(\hat{\fR})$ are elevations of $Z_1,Z_2$ in $\cP_j(\fR)$, then there are elevations $\wt{Z}_1,\wt{Z}_2$ to $\wt{X}$ of $\hat{Z}_1,\hat{Z}_2$ respectively with $\hH(\wt{Z}_1)\subset\hH(\wt{Z}_2)$; but $\wt{Z}_1,\wt{Z}_2$ are also elevations of $Z_1,Z_2$, so $Z_1\leq Z_2$.	
	It follows that the map $\cP_j(\hat{\fR})\to\cP_j(\fR)$ preserves both $\leq$ and $\sim$, so it induces a $\leq$-preserving map $\cP_j(\hat{\fR})/\sim\to\cP_j(\fR)/\sim$.
	
	Finally, if $\hat{Z}_1\leq\hat{Z}_2$ and $Z_1\leq Z_2$ are as above, but with $\hat{Z}_1\nsim\hat{Z}_2$, then there must be elevations $\wt{Z}_1,\wt{Z}_2$ to $\wt{X}$ that satisfy a strict inclusion $\hH(\wt{Z}_1)\subsetneq\hH(\wt{Z}_2)$.
	By Lemma \ref{lem:Z1Z2Z1}, $Z_1\nsim Z_2$.
	Therefore the map $\cP_j(\hat{\fR})/\sim\to\cP_j(\fR)/\sim$ preserves the lengths of $\leq$-chains.
\end{proof}

\subsection{Property (Hyp$_j$)}\label{subsec:Hypj}

The following property will be important for the arguments in Section \ref{subsec:trap}.

\begin{defn}(Property (Hyp$_j$))\\\label{defn:Hypj}
	Let $\fR=(y_0,Y_1,y_1,Y_2,y_2,\dots,Y_n,y_n)$ be a route in a finite directly special cube complex $X$, with $n\geq4$. For $2\leq j\leq n-1$, we say that $\fR$ satisfies property (Hyp$_j$) if, for every elevation $\wt{\fR}=(\ti{y}_0,\wt{Y}_1,\ti{y}_1,\wt{Y}_2,\ti{y}_2,\dots,\wt{Y}_n,\ti{y}_n)$ of $\fR$ to the universal cover $\wt{X}\to X$, we have $\hH(\wt{Y}_1)\cap\hH(\wt{Y}_j)=\cap_{i=1}^{j}\hH(\wt{Y}_i)$ (remember that $\hH(\wt{A})$ is the collection of hyperplanes that intersect $\wt{A}$ -- see Definition \ref{defn:convexhulls}).
\end{defn}

The aim of this section is to reduce Theorem \ref{thm:embroutes} to the case where the route $\fR$ satisfies property (Hyp$_j$) for all $2\leq j\leq n-1$ (see Lemma \ref{lem:inductHyp}).
The method will be to replace a given route $\fR$ with a finite collection of routes $\Omega$ that satisfy property (Hyp$_j$) for more values of $j$.
The strategy for this replacement will be to consider an elevation $\wt{\fR}$ of $\fR$ to the universal cover $\wt{X}$ of $X$, modify $\wt{\fR}$ using the projection map to a certain convex subcomplex of $\wt{X}$, and then descend the resulting route to a route in $X$.
This will involve results from Section \ref{sec:projections} and also the complexities $\kappa_j(\fR)$ from Section \ref{subsec:kappaj}.

\begin{defn}(The route $\wt{\fR}_{proj}$)\\\label{defn:wtfRproj}
	Let $\fR=(y_0,Y_1,y_1,Y_2,y_2,\dots,Y_n,y_n)$ be an essential embedded closed route in a finite directly special cube complex $X$, with $n\geq4$. Fix $3\leq j\leq n-1$.	
	Let $\wt{\fR}=(\ti{y}_0,\wt{Y}_1,\ti{y}_1,\wt{Y}_2,\ti{y}_2,\dots,\wt{Y}_n,\ti{y}_n)$ be an elevation of $\fR$ to the universal cover $\wt{X}\to X$.
	Suppose $\wt{Z}=\Pi_{\wt{Y}_1}(\wt{Y}_j)$ is an elevation of a subcomplex $Z\in\cP_j(\fR)$, with $[Z]_j$ $\leq$-maximal.
	
	Let $\wt{W}$ be an orthogonal complement of $\wt{Z}$.
	Recall that we have an isomorphism $\phi:\wt{Z}\times\wt{W}\to\Hull(\wt{Z}\cup\wt{W})$.
	For $2\leq l<j$, define $\wt{W}_l=\Pi_{\wt{W}}(\wt{Y}_l)$ and $\wt{P}_l=\phi(\wt{Z}\times\wt{W}_l)$.
	Then for $1\leq l<j$, define $\ti{x}_l=\Pi_{\wt{W}}(\ti{y}_l)$.
	This gives us a route
	\begin{equation}\label{wtfRproj}
		\wt{\fR}_{proj}=\wt{\fR}_{proj}(\wt{\fR},j)=(\ti{y}_0,\wt{Y}_1,\ti{x}_1,\wt{P}_2,\ti{x}_2,\dots,\wt{P}_{j-1},\ti{x}_{j-1},\wt{Y}_j,\ti{y}_j,\dots,\wt{Y}_n,\ti{y}_n)
	\end{equation}
	in $\wt{X}$ (see Figure \ref{fig:wtRproj}). Note that we really have $\ti{x}_1\in\wt{Y}_1$ because $\ti{x}_1\in\Pi_{\wt{W}}(\wt{Y}_1)=\wt{Y}_1\cap\wt{W}$ (and $\wt{Y}_1\cap\wt{W}\supset\wt{Z}\cap\wt{W}\neq\emptyset$).
	Similarly, $\ti{x}_{j-1}\in\wt{Y}_j$ because $\ti{x}_{j-1}\in\Pi_{\wt{W}}(\wt{Y}_j)=\wt{Y}_j\cap\wt{W}$ (and $\wt{Y}_j\cap\wt{W}\neq\emptyset$ because one of the slices $\phi(\wt{Z}\times\{\ti{w}\})$ is contained in $\wt{Y}_j$ by Theorem \ref{thm:bridge}).
\end{defn}

\begin{figure}[H]
	\centering
	\scalebox{0.8}{
		\begin{tikzpicture}[auto,node distance=2cm,
			thick,every node/.style={circle,draw,font=\small},
			every loop/.style={min distance=2cm},
			hull/.style={draw=none},
			]
			\tikzstyle{label}=[draw=none,font=\LARGE]

			\draw[rounded corners=20pt,Green] (0,0) rectangle (3,9) {};
			\draw[rounded corners=20pt,red] (2,1) rectangle (5,3) {};
			\draw[rounded corners=20pt,blue] (4,1) rectangle (6,3) {};
			\draw[rounded corners=20pt,orange] (8,1) rectangle (10,3) {};
			\draw[rounded corners=20pt,cyan] (9,0) rectangle (12,9) {};
			\draw[rounded corners=20pt,magenta] (11,1) rectangle (13,3) {};
			\draw[rounded corners=20pt,brown] (15,1) rectangle (17,3) {};
			
			\draw (1,6) rectangle (15,8);
			\draw[violet,line width=5pt] (3,6)--(3,8);
			\draw[darkgray,line width=5pt] (1,6)--(15,6);
			\draw[red] (2,6.2) rectangle (5,7.9);
			\draw[blue] (4,6.2) rectangle (6,7.9);
			\draw[orange] (8,6.2) rectangle (10,7.9);
			
			\path (6.5,2) edge [dashed] (7.5,2);
			\path (6.5,7) edge [dashed] (7.5,7);
			\path (13.5,2) edge [dashed] (14.5,2);
			
			\draw[draw=black,fill=blue,-triangle 90, ultra thick] (6,3.7) -- (6,5.2);
			
			\node[label] (Pi) at (6.8,4.3) {$\Pi_{\wt{W}}$};
			\node[label,darkgray] (W) at (7,5.5) {$\wt{W}$};
			\node[label,violet] (Z) at (2.5,7) {$\wt{Z}$};
			\node[label,Green] (Y1) at (1.5,9.5) {$\wt{Y}_1$};
			\node[label,cyan] (Yj) at (10.5,9.5) {$\wt{Y}_j$};
			\node[label,red] (P2) at (3.5,8.5) {$\wt{P}_2$};
			\node[label,blue] (P3) at (5,8.5) {$\wt{P}_3$};
			\node[label,orange] (Pj-1) at (8.3,8.5) {$\wt{P}_{j-1}$};
			\node[label,red] (Y2) at (3.5,3.5) {$\wt{Y}_2$};
			\node[label,blue] (Y3) at (5,3.5) {$\wt{Y}_3$};
			\node[label,orange] (Yj-1) at (8.3,3.5) {$\wt{Y}_{j-1}$};
			\node[label,magenta] (Yj+1) at (12.7,3.5) {$\wt{Y}_{j+1}$};
			\node[label,brown] (Yn) at (16,3.5) {$\wt{Y}_n$};
			
			\node[fill] at (2.5,6.2){};
			\node[fill] at (4.5,6.2){};
			\node[fill] at (9.5,6.2){};
			\node[fill] at (2.5,2){};
			\node[fill] at (4.5,2){};
			\node[fill] at (9.5,2){};
			\node[fill] at (0.5,2){};
			\node[fill] at (11.5,2){};
			\node[fill] at (16.5,2){};
			
			\node[label] (x1) at (2.5,5.5){$\ti{x}_1$};
			\node[label] (x2) at (4.5,5.5){$\ti{x}_2$};
			\node[label] (xj-1) at (9.5,5.5){$\ti{x}_{j-1}$};
			
			\node[label] (y0) at (0.5,1.5){$\ti{y}_0$};
			\node[label] (y1) at (2.5,1.5){$\ti{y}_1$};
			\node[label] (y2) at (4.5,1.5){$\ti{y}_2$};
			\node[label] (yj-1) at (9.5,1.5){$\ti{y}_{j-1}$};
			\node[label] (yj) at (11.5,1.5){$\ti{y}_j$};
			\node[label] (yn) at (16.5,1.5){$\ti{y}_n$};

		\end{tikzpicture}
	}
	\caption{The routes $\wt{\fR}$ and $\wt{\fR}_{proj}$.}\label{fig:wtRproj}
\end{figure}
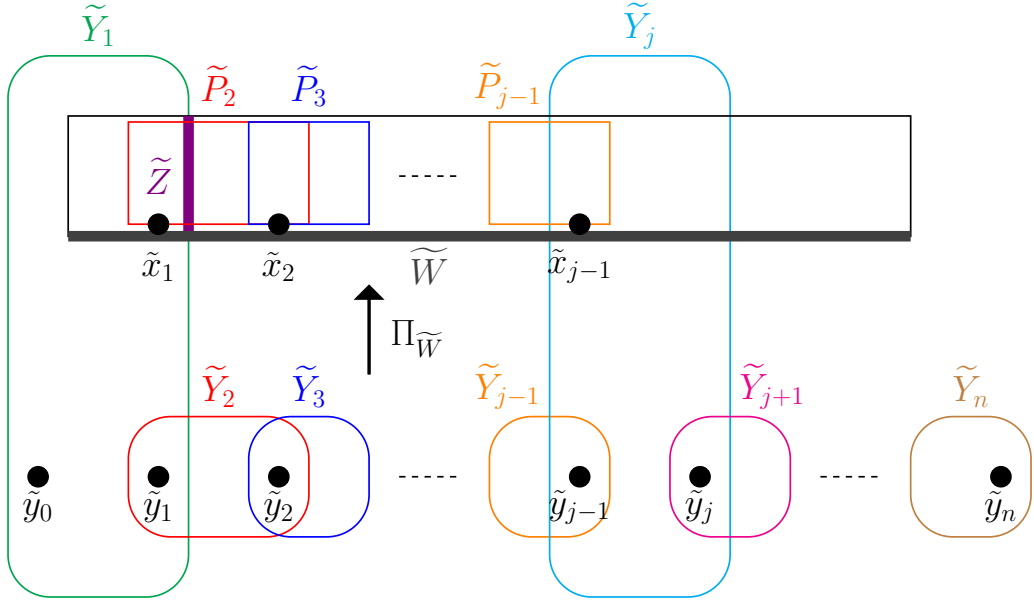

\begin{lem}\label{lem:fRproj}
The route $\wt{\fR}_{proj}(\wt{\fR},j)$ is the elevation of a finite closed route $\fR_{proj}=\fR_{proj}(\wt{\fR},j)$ in $X$.
Moreover, the sizes of the complexes in $\fR_{proj}$ are bounded in terms of $X$ (each is either a subcomplex of $X$, or isomorphic to a product of two subcomplexes of $X$).
\end{lem}
\begin{proof}
Firstly, it follows from Proposition \ref{prop:orthcomp}\ref{item:PApreserved} that $G_{\wt{Z}}$ and $G_{\wt{W}}$ preserve $\phi(\wt{Z}\times\wt{W})$, moreover $G_{\wt{Z}}$ (resp. $G_{\wt{W}}$) acts on the $\wt{Z}$ factor (resp. $\wt{W}$ factor) in the natural way and it acts trivially on the other factor.
Secondly, Proposition \ref{prop:orthcomp}\ref{item:Bembed} tells us that $\wt{W}/G_{\wt{W}}$ embeds in $X$ (remember that direct specialness is a stronger property than weak specialness). 
Thirdly, Proposition \ref{prop:projection} implies that $G_{\wt{W}_l}=G_{\wt{W}}\cap G_{\wt{Y}_l}$ (for $2\leq l<j$).
Hence $G_{\wt{Z}}$ and $G_{\wt{W}_l}$ both preserve $\wt{P}_l=\phi(\wt{Z}\times\wt{W}_l)$.
Moreover, Proposition \ref{prop:projection} implies that $\wt{W}_l/G_{\wt{W}_l}$ embeds in $X$, so we have that 
\begin{equation}\label{Plquotient}
	\wt{P}_l/G_{\wt{Z}}G_{\wt{W}_l}\cong \wt{Z}/G_{\wt{Z}}\times \wt{W}_l/G_{\wt{W}_l}\cong Z\times  \wt{W}_l/G_{\wt{W}_l}
\end{equation}
is a finite complex.
Now put
\begin{equation}
	\fR_{proj}(\wt{\fR},j)=(y_0,Y_1,x_1,\wt{P}_2/G_{\wt{Z}}G_{\wt{W}_2},x_2,\dots,\wt{P}_{j-1}/G_{\wt{Z}}G_{\wt{W}_{j-1}},x_{j-1},Y_j,y_j,\dots,Y_n,y_n).
\end{equation}
The vertices $x_l$ are the images of $\ti{x}_l$ in $X$, and the maps $\wt{P}_l/G_{\wt{Z}}G_{\wt{W}_l}\to\wt{X}/G=X$ are the natural quotient maps.
Clearly, $\fR_{proj}(\wt{\fR},j)$ is a finite closed route in $X$.
Moreover, (\ref{Plquotient}) implies that each complex in $\fR_{proj}(\wt{\fR},j)$ is either a subcomplex of $X$, or isomorphic to a product of two subcomplexes of $X$.
\end{proof}

\begin{lem}\label{lem:3propfRproj}
	The route $\fR_{proj}=\fR_{proj}(\wt{\fR},j)$ satisfies the following properties:
	\begin{enumerate}
		\item\label{item:isessential} $\fR_{proj}$ is essential.
		\item\label{item:Hypj'} $\fR_{proj}$ satisfies $\mathrm{(Hyp}_j\mathrm{)}$.
		\item\label{item:Hypl'} If $\fR$ (from Definition \ref{wtfR'proj}) satisfies $\mathrm{(Hyp}_l\mathrm{)}$ for some $2\leq l<j$, then $\fR_{proj}$ also satisfies $\mathrm{(Hyp}_l\mathrm{)}$.
	\end{enumerate}
\end{lem}
\begin{proof}
By Lemma \ref{lem:elevuniv}, every elevation of $\fR_{proj}$ to $\wt{X}$ takes the form
\begin{equation}\label{wtfR'proj}
	\wt{\fR}'_{proj}=(g_0\ti{y}_0,g_0\wt{Y}_1,g_1\ti{x}_1,g_1\wt{P}_2,g_2\ti{x}_2,\dots,g_{j-2}\wt{P}_{j-1},g_{j-1}\ti{x}_{j-1},g_{j-1}\wt{Y}_j,g_j\ti{y}_j,\dots,g_{n-1}\wt{Y}_n,g_n\ti{y}_n),
\end{equation}
where $g_0\in G$, and $g_i=g_0k_1k_2\cdots k_i$ ($1\leq i\leq n$), with 
\begin{equation}
	k_1\in G_{\wt{Y}_1}, k_2\in G_{\wt{Z}}G_{\wt{W}_2},\dots, k_{j-1}\in G_{\wt{Z}}G_{\wt{W}_{j-1}}, k_j\in G_{\wt{Y}_j},\dots, k_n\in G_{\wt{Y}_n}.
\end{equation}
For $2\leq l<j$, let $k_l=a_lk_l^*$ with $a_l\in G_{\wt{Z}}$ and $k_l^*\in G_{\wt{W}_l}$.
Define $k_1^*=k_1a_2a_3\cdots a_{j-1}$. As $G_{\wt{Z}}=G_{\wt{Y}_1}\cap G_{\wt{Y}_j}$ (Proposition \ref{prop:projection}), we have $k_1^*\in  G_{\wt{Y}_1}$.
Similarly, $G_{\wt{W}_l}=G_{\wt{W}}\cap G_{\wt{Y}_l}$, so $k_l^*\in G_{\wt{Y}_l}$ ($2\leq l<j$).
Putting $k_i^*=k_i$ for $j\leq i\leq n$, and $g_i^*=g_0k_1^*k_2^*\cdots k_i^*$ for $1\leq i\leq n$, we can then apply Lemma \ref{lem:elevuniv} again to build an elevation $\wt{\fR}'$ of $\fR$ to $\wt{X}$ defined by
\begin{equation}\label{wtR'}
	\wt{\fR}'=(g_0\ti{y}_0,g_0\wt{Y}_1,g_1^*\ti{y}_1,g_1^*\wt{Y}_2,g_2^*\ti{y}_2,\dots,g_{n-1}^*\wt{Y}_n,g_n^*\ti{y}_n).
\end{equation}

We now verify properties \ref{item:isessential}--\ref{item:Hypl'} of the lemma:
\begin{enumerate}
	\item By Proposition \ref{prop:orthcomp}\ref{item:ABcommute}, $G_{\wt{Z}}$ commutes with $G_{\wt{W}_2},\dots, G_{\wt{W}_{j-1}}$, so
	\begin{align}\label{gngn*}
		g_n^*&=g_0k_1^*k_2^*\cdots k_n^*\\\nonumber
		&=g_0k_1a_2a_3\cdots a_{j-1}k_2^*\cdots k_{j-1}^* k_j\cdots k_n\\\nonumber
		&=g_0k_1(a_2k_2^*)\cdots(a_{j-1}k_{j-1}^*)k_j\cdots k_n\\\nonumber
		&=g_0k_1k_2\cdots k_n\\\nonumber
		&=g_n.\nonumber
	\end{align}
	Lemma \ref{lem:routeelev}\ref{item:essential} says that a route in $X$ is essential if and only if it has no closed elevations to $\wt{X}$.
	The route $\fR$ is essential by assumption, so $\wt{\fR}'$ is not closed, but the routes $\wt{\fR}'$ and $\wt{\fR}'_{proj}$ have the same terminal and initial vertices by (\ref{gngn*}), thus $\wt{\fR}'_{proj}$ is not closed either. Since $\wt{\fR}'_{proj}$ represents an arbitrary elevation of $\fR_{proj}$ to $\wt{X}$, we conclude that $\fR_{proj}$ is essential.
	
	\item Again consider the elevation $\wt{\fR}'_{proj}$ of $\fR_{proj}$ given in (\ref{wtfR'proj}).
	Our task is to show that
	\begin{equation}\label{Hypj}
		\hH(g_0\wt{Y}_1)\cap\hH(g_{j-1}\wt{Y}_j)\subset\bigcap_{l=2}^{j-1}\hH(g_{l-1}\wt{P}_l).
	\end{equation} 
	The elements $k_2,\dots,k_{j-1}$ are all in $G_{\wt{Z}}G_{\wt{W}}$, so they all preserve $\phi(\wt{Z}\times\wt{W})$ without interchanging the two factors in the product (Proposition \ref{prop:orthcomp}), hence they all stabilize the collection of hyperplanes $\hH(\wt{Z})=\hH(\wt{Y}_1)\cap\hH(\wt{Y}_j)$ (this last equality is due to Theorem \ref{thm:bridge}).
	As a result,
		\begin{equation*}
		\hH(\wt{Z})=\hH(\wt{Y}_1)\cap\hH(\wt{Y}_j)=\hH(k_2k_3\cdots k_{j-1}\wt{Y}_1)\cap\hH(k_2k_3\cdots k_{j-1}\wt{Y}_j).
	\end{equation*}
	Thus
	\begin{equation*}
		\hH(\wt{Z})\subset\hH(\wt{Y}_1)\cap\hH(k_2k_3\cdots k_{j-1}\wt{Y}_j)=\hH(\Pi_{\wt{Y}_1}(k_2k_3\cdots k_{j-1}\wt{Y}_j)),
	\end{equation*}
	with the equality again due to Theorem \ref{thm:bridge}.
	Left multiplying by $g_0k_1$, we get
	\begin{equation}\label{hypinc}
		\hH(g_1\wt{Z})\subset\hH(g_0\wt{Y}_1)\cap\hH(g_{j-1}\wt{Y}_j)=\hH(\Pi_{g_0\wt{Y}_1}(g_{j-1}\wt{Y}_j)).
	\end{equation}
	A similar computation to (\ref{gngn*}) shows that $g_{j-1}=g_{j-1}^*$, so $g_0\wt{Y}_1$ and $g_{j-1}\wt{Y}_j$ are the 1st and $j$-th subcomplexes in the route $\wt{\fR}'$. Since $\wt{\fR}'$ is an elevation of $\fR$, we deduce from Remark \ref{remk:iselevation} that $\Pi_{g_0\wt{Y}_1}(g_{j-1}\wt{Y}_j)$ is an elevation of some complex in $\cP_j(\fR)$.
	But we assumed $[Z]_j$ is $\leq$-maximal, so Lemma \ref{lem:leqmaximal} implies that the inclusion in (\ref{hypinc}) is actually an equality.
	Since $k_2,\dots,k_{j-1}$ stabilize $\hH(\wt{Z})$ and $\hH(\wt{Z})\subset\hH(\wt{P}_l)$ for $2\leq l\leq j-1$, it follows that
	\begin{equation*}
		\hH(g_0\wt{Y}_1)\cap\hH(g_{j-1}\wt{Y}_j)=\hH(g_1\wt{Z})\subset\bigcap_{l=2}^{j-1}\hH(g_1k_2k_3\cdots k_{l-1}\wt{P}_l)=\bigcap_{l=2}^{j-1}\hH(g_{l-1}\wt{P}_l).
	\end{equation*}
	
	\item We now suppose there is $2\leq l<j$ such that $\fR$ satisfies (Hyp$_l$). Our task is to show that $\fR_{proj}$ also satisfies (Hyp$_l$).
	Again consider the elevation $\wt{\fR}'_{proj}$ of $\fR_{proj}$ given in (\ref{wtfR'proj}).
	We must show that
	\begin{equation}\label{Hypl}
		\hH(g_0\wt{Y}_1)\cap\hH(g_{l-1}\wt{P}_l)\subset\bigcap_{i=2}^{l}\hH(g_{i-1}\wt{P}_i).
	\end{equation} 
	For $2\leq i\leq l$, we can make a similar computation to (\ref{gngn*})
	\begin{align}\label{gi-1*gi-1}
		g_{i-1}^*&=g_0k_1^*k_2^*\cdots k_{i-1}^*\\\nonumber
		&=g_0k_1a_2a_3\cdots a_{j-1}k_2^*\cdots k_{i-1}^*\\\nonumber
		&=g_0k_1(a_2k_2^*)\cdots (a_{i-1}k_{i-1}^*)a_i\cdots a_{j-1}\\\nonumber
		&=g_0k_1k_2\cdots k_{i-1}a_i\cdots a_{j-1}\\\nonumber
		&=g_{i-1}a_i\cdots a_{j-1}.\\\nonumber
	\end{align}
	As $a_i\cdots a_{j-1}\in G_{\wt{Z}}$ stabilize $\phi(\wt{Z}\times\wt{W})$, acting trivially on the $\wt{W}$ factor, we deduce that $a_i\cdots a_{j-1}$ stabilize the sets of hyperplanes $\hH(\wt{W})$ and $\hH(\wt{W}_i)$. 
	Combined with (\ref{gi-1*gi-1}), we get
	\begin{equation}\label{hHWi}
		\hH(g_{i-1}^*\wt{W})=\hH(g_{i-1}\wt{W})\quad\text{and}\quad\hH(g_{i-1}^*\wt{W}_i)=\hH(g_{i-1}\wt{W}_i).
	\end{equation}
	Furthermore, $g_{i-1}=g_1k_2\cdots k_{i-1}$, and $k_2\cdots k_{i-1}$ also stabilize $\hH(\wt{W})$, so we additionally have
	\begin{equation}\label{hHW}
		\hH(g_{i-1}^*\wt{W})=\hH(g_{i-1}\wt{W})=\hH(g_1\wt{W}).
	\end{equation}
	As $\fR$ satisfies (Hyp$_l$) and $\wt{\fR}'$ (given in (\ref{wtR'})) is an elevation of $\fR$ to $\wt{X}$, we have
	\begin{equation}\label{HyplforR}
		\hH(g_0\wt{Y}_1)\cap\hH(g_{l-1}^*\wt{Y}_l)\subset\bigcap_{i=2}^{l}\hH(g_{i-1}^*\wt{Y}_i).
	\end{equation}
	Using (\ref{hHWi}), (\ref{hHW}), (\ref{HyplforR}) and the fact that $\hH(\wt{W}_i)=\hH(\wt{W})\cap\hH(\wt{Y}_i)$ (Theorem \ref{thm:bridge}), we deduce
	\begin{align}\label{bigcaphHWi}
		\hH(g_0\wt{Y}_1)\cap\hH(g_{l-1}\wt{W}_l)&=\hH(g_0\wt{Y}_1)\cap\hH(g_{l-1}^*\wt{W}_l)\\\nonumber
		&=\hH(g_0\wt{Y}_1)\cap\hH(g_{l-1}^*\wt{W})\cap\hH(g_{l-1}^*\wt{Y}_l)\\\nonumber
		&=\hH(g_0\wt{Y}_1)\cap\hH(g_1\wt{W})\cap\hH(g_{l-1}^*\wt{Y}_l)\\\nonumber
		&\subset\bigcap_{i=2}^l(\hH(g_{i-1}^*\wt{W})\cap\hH(g_{i-1}^*\wt{Y}_i))\\\nonumber
		&=\bigcap_{i=2}^l\hH(g_{i-1}^*\wt{W}_i)\\\nonumber
		&=\bigcap_{i=2}^l\hH(g_{i-1}\wt{W}_i).\\\nonumber
	\end{align}
	For $2\leq i\leq l$, we know that $\wt{P}_i=\phi(\wt{Z}\times\wt{W}_i)$, so $\hH(\wt{P}_i)=\hH(\wt{Z})\cup\hH(\wt{W}_i)$.
	And $k_i$ stabilizes $\wt{P}_i=\phi(\wt{Z}\times\wt{W}_i)$ without swapping the factors, so it stabilizes $\hH(\wt{Z})$. Thus $\hH(g_{i-1}\wt{Z})=g_1k_2\cdots k_{i-1}\hH(\wt{Z})=g_1\hH(\wt{Z})$.
	Finally, combining these two facts with (\ref{bigcaphHWi}) yields
	\begin{align*}
		\hH(g_0\wt{Y}_1)\cap\hH(g_{l-1}\wt{P}_l)&\subset (\hH(g_0\wt{Y}_1)\cap\hH(g_{l-1}\wt{W}_l))\cup \hH(g_{l-1}\wt{Z})\\\nonumber
		&=(\hH(g_0\wt{Y}_1)\cap\hH(g_{l-1}\wt{W}_l))\cup g_1\hH(\wt{Z})\\\nonumber
		&\subset \left(\bigcap_{i=2}^l\hH(g_{i-1}\wt{W}_i)\right)\cup g_1\hH(\wt{Z})\\\nonumber
		&=\bigcap_{i=2}^l(\hH(g_{i-1}\wt{W}_i)\cup \hH(g_{i-1}\wt{Z}))\\\nonumber
		&=\bigcap_{i=2}^l\hH(g_{i-1}\wt{P}_i).\qedhere
	\end{align*}
\end{enumerate}
\end{proof}

Putting together Lemmas \ref{lem:fRproj} and \ref{lem:3propfRproj} yields the following.

\begin{lem}\label{lem:Omega}
Let $\fR=(y_0,Y_1,y_1,Y_2,y_2,\dots,Y_n,y_n)$ be an essential embedded closed route in a finite directly special cube complex $X$, with $n\geq4$. Fix $2\leq j\leq n-1$.
There exists a finite collection $\Omega=\Omega(\fR,j)$ of finite essential closed routes (not necessarily embedded) in $X$ of length $n$ such that:
\begin{enumerate}
	\item\label{item:Hypj} Each route in $\Omega$ satisfies $\mathrm{(Hyp}_j\mathrm{)}$.
	\item\label{item:Hypl} If $\fR$ satisfies $\mathrm{(Hyp}_l\mathrm{)}$ for some $2\leq l<j$, then the routes in $\Omega$ also satisfy $\mathrm{(Hyp}_l\mathrm{)}$.
	\item\label{item:matchOmega} If an elevation $\wt{\fR}$ of $\fR$ to the universal cover $\wt{X}\to X$ has $\Pi_{\wt{Y}_1}(\wt{Y}_j)$ as an elevation of some $Z\in\cP_j(\fR)$ with $[Z]_j$ $\leq$-maximal, then $\wt{\fR}$ has the same initial and terminal vertices as some elevation of a route in $\Omega$.
\end{enumerate} 
\end{lem}
\begin{proof}
	If $j=2$ then we can just take $\Omega=\{\fR\}$. Property \ref{item:Hypj} holds because any route in $X$ trivially satisfies (Hyp$_2$). Properties \ref{item:Hypl} and \ref{item:matchOmega} also hold for trivial reasons. 
	
	Now assume $j\geq3$. We take $\Omega=\Omega(\fR,j)$ to be the collection of all routes $\fR_{proj}(\wt{\fR},j)$ (as constructed in Lemma \ref{lem:fRproj}), where $\wt{\fR}$ ranges over all elevations $\wt{\fR}$ of $\fR$ to $\wt{X}$ such that $\Pi_{\wt{Y}_1}(\wt{Y}_j)$ is an elevation of a subcomplex $Z\in\cP_j(\fR)$ with $[Z]_j$ $\leq$-maximal (as in Definition \ref{defn:wtfRproj}).
	For such a route $\wt{\fR}$, we know that $\wt{\fR}_{proj}(\wt{\fR},j)$ has the same initial and terminal vertices as $\wt{\fR}$ (see (\ref{wtfRproj})), and $\wt{\fR}_{proj}(\wt{\fR},j)$ is an elevation of $\fR_{proj}(\wt{\fR},j)$ (Lemma \ref{lem:fRproj}), so part \ref{item:matchOmega} is satisfied. Parts \ref{item:Hypj} and \ref{item:Hypl} (and essentialness of the routes in $\Omega$) follow from Lemma \ref{lem:3propfRproj}.
	Finally, the finiteness of $\Omega$ follows from the finiteness of $X$ and the second part of Lemma \ref{lem:fRproj}.	
\end{proof}

As promised, we can now reduce Theorem \ref{thm:embroutes} to the case where the route $\fR$ satisfies property (Hyp$_j$) for all $2\leq j\leq n-1$.

\begin{lem}\label{lem:inductHyp}
	Fix $n\geq4$.
	Suppose Theorem \ref{thm:embroutes} holds for routes of length $n$ that satisfy
	property (Hyp$_j$) for all $2\leq j\leq n-1$.
	Then Theorem \ref{thm:embroutes} holds for all routes of length $n$.
\end{lem}
\begin{proof}
We proceed by a double induction, inducting first on the number 
\begin{equation}
	\mathrm{Hyp}(\fR)=\max\{j\mid \fR\text{ satisfies (Hyp$_l$) for all $2\leq l\leq j$}\}.
\end{equation}
This first induction will go in a backwards direction, with $\mathrm{Hyp}(\fR)=n-1$ being the base case (Theorem \ref{thm:embroutes} holds in this case by assumption of the lemma), and $2\leq\mathrm{Hyp}(\fR)\leq n-1$ the general case (note that we always have  $2\leq\mathrm{Hyp}(\fR)$ since any route satisfies (Hyp$_2$)). 
The second induction will be on the sum of complexities $\sum_{j=2}^{n-1}\kappa_j(\fR)$ (and will go in a forwards direction).

If $\fR$ satisfies (Hyp$_j$), then any elevation $\hat{\fR}$ of $\fR$ to a finite cover of $X$ will also satisfy (Hyp$_j$). This is simply because the set of elevations of $\hat{\fR}$ to the universal cover $\wt{X}\to X$ is a subset of the set of elevations of $\fR$ to $\wt{X}$.
As a result, the number $\mathrm{Hyp}(\fR)$ can only increase when replacing $\fR$ by an elevation in a finite cover.
Similarly, Lemma \ref{lem:kappadown} implies that the complexities $\kappa_j(\fR)$ can only decrease when passing to a finite cover.

Now suppose $\mathrm{Hyp}(\fR)=j-1<n-1$.
Apply Lemma \ref{lem:Omega} to produce the finite collection $\Omega=\Omega(\fR,j)$ of finite essential closed routes in $X$.
Parts \ref{item:Hypj} and \ref{item:Hypl} of Lemma \ref{lem:Omega} ensure that each $\fR_{proj}\in\Omega$ satisfies $\mathrm{Hyp}(\fR_{proj})\geq j$.
By Proposition \ref{prop:routeembed}, there is a finite cover $X'\to X$ where all elevations of routes in $\Omega$ embed. Each such elevation $\fR'_{proj}$ also satisfies $\mathrm{Hyp}(\fR'_{proj})\geq j$, so Theorem \ref{thm:embroutes} holds for the routes $\fR'_{proj}$ by induction. Hence there is a further finite cover $\hat{X}\to X'$ with no closed elevation of any of the routes $\fR'_{proj}$. This implies that $\hat{X}$ has no closed elevation of any of the routes in $\Omega$ (using Lemma \ref{lem:elevelev}\ref{item:R''toR'}).

We now claim that any closed elevation $\hat{\fR}$ of $\fR$ to $\hat{X}$ satisfies $\kappa_j(\hat{\fR})<\kappa_j(\fR)$.
Suppose for contradiction that $\kappa_j(\hat{\fR})=m=\kappa_j(\fR)$.
Let $[\hat{Z}_1]_j\lneq[\hat{Z}_2]_j\lneq\cdots\lneq[\hat{Z}_m]_j$ be a chain in $\cP_j(\hat{\fR})/\sim$.
By Lemma \ref{lem:kappadown}, this descends to a chain $[Z_1]\lneq[Z_2]\lneq\cdots\lneq[Z_m]$ in $\cP_j(\fR)$.
Since $m=\kappa_j(\fR)$, we know that $[Z_m]_j$ is $\leq$-maximal.
Let $\wt{\fR}=(\ti{y}_0,\wt{Y}_1,\ti{y}_1,\wt{Y}_2,\ti{y}_2,\dots,\wt{Y}_n,\ti{y}_n)$ be an elevation of $\hat{\fR}$ to the universal cover $\wt{X}\to X$ with $\Pi_{\wt{Y}_1}(\wt{Y}_j)$ an elevation of $\hat{Z}_m$ (hence also an elevation of $Z_m$).
By Lemma \ref{lem:Omega}\ref{item:matchOmega}, $\wt{\fR}$ has the same initial and terminal vertices as some elevation $\wt{\fR}_{proj}$ of a route $\fR_{proj}\in\Omega$.
Let $\hat{\fR}_{proj}$ be the (unique) elevation of $\fR_{proj}$ to $\hat{X}$ such that $\wt{\fR}_{proj}$ is an elevation of $\hat{\fR}_{proj}$ (Lemma \ref{lem:elevelev}).
Then $\hat{\fR}_{proj}$ has the same initial and terminal vertices as $\hat{\fR}$. But $\hat{\fR}$ is a closed route by assumption, so $\hat{\fR}_{proj}$ is a closed elevation of $\fR_{proj}$ to $\hat{X}$, in contradiction with the construction of $\hat{X}$.



By induction, Theorem \ref{thm:embroutes} holds for each closed elevation of $\fR$ to $\hat{X}$ (if any even exist), thus there exists a finite cover of $\hat{X}$ with no closed elevations of $\fR$.
\end{proof}

\subsection{Property (Trap)}\label{subsec:trap}

The following property will be crucial in the next section when we apply the Walker and Imitator construction.

\begin{defn}(Property (Trap))\\\label{defn:trap}
Let $\fR=(y_0,Y_1,y_1,Y_2,y_2,\dots,Y_n,y_n)$ be an embedded route in a finite non-positively curved cube complex $X$ with $n\geq4$.
We say that $\fR$ satisfies property (Trap) if, for any vertex $y\in Y_1\cap Y_2$ and edge $e$ in $Y_1$ incident to $y$ whose interior lies outside of $Y_2$, the hyperplane $H=H(e)$ is disjoint from $Y_2,Y_3,\dots,Y_{n-1}$.
(The reason for the name (Trap) will become clear in the next section!) A picture of a route $\fR$ that fails property (Trap) is given in Figure \ref{fig:notrap}.
\end{defn}

\begin{figure}[H]
	\centering
	\scalebox{0.8}{
		\begin{tikzpicture}[auto,node distance=2cm,
			thick,every node/.style={circle,draw,font=\small},
			every loop/.style={min distance=2cm},
			hull/.style={draw=none},
			]
			\tikzstyle{label}=[draw=none,font=\LARGE]

			\draw[rounded corners=20pt] (0,0) rectangle (3,6) {};
			\draw[rounded corners=20pt] (1,1) rectangle (5,3) {};
			\draw[rounded corners=20pt] (4,1) rectangle (7,3) {};
			\draw[rounded corners=20pt] (9,1) rectangle (12,3) {};
			\draw[rounded corners=20pt] (14,1) rectangle (17,3) {};

			\node[fill] at (2,3){};
			\node[fill] at (2,4){};
			\path[line width=2pt] (2,3) edge (2,4);
			\draw[rounded corners=30pt,red] (-1,3.5)--(6,3.5)--
			(6,5.5)--(10.5,5.5)--(10.5,0);
			
			\node[label] (e) at (2.3,4.3){$e$};
			\node[label] (Y1) at (1.5,5.4){$Y_1$};
			\node[label] (Y2) at (3.5,2){$Y_2$};
			\node[label] (Y3) at (6,2){$Y_3$};
			\node[label] (Yi) at (9.8,2){$Y_i$};
			\node[label] (Yn) at (15.5,2){$Y_n$};
			\node[label,red] (H) at (5.5,4.5){$H$};
			
			\path (7.5,2) edge [dashed] (8.5,2);
			\path (12.5,2) edge [dashed] (13.5,2);

		\end{tikzpicture}
	}
	\caption{Picture of a route $\fR$ that fails property (Trap).}\label{fig:notrap}
\end{figure}
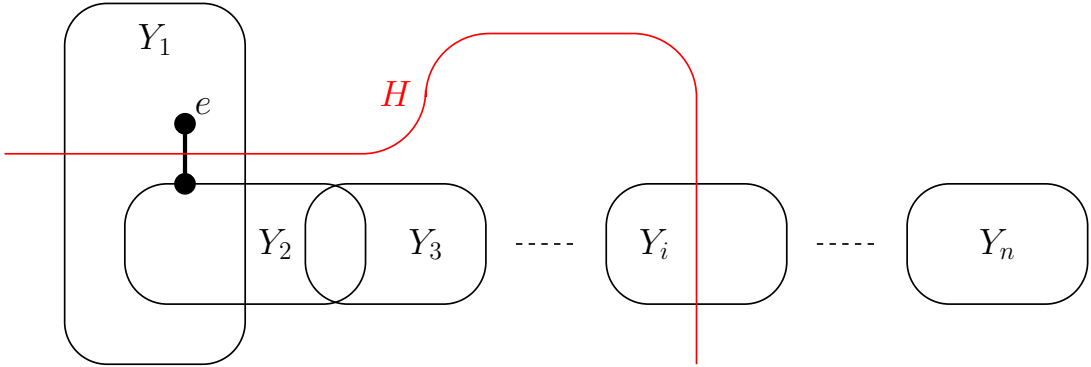

\begin{lem}\label{lem:propTrap}
Let $\fR$ be an essential embedded closed route of length $n\geq4$ in a finite directly special cube complex $X$.
Suppose $\fR$ satisfies property (Hyp$_j$) for all $2\leq j\leq n-1$, and suppose Theorem \ref{thm:embroutes} holds for routes of length less than $n$.
Then there is a finite cover $\hat{X}\to X$ such that each closed elevation of $\fR$ to $\hat{X}$ satisfies property (Trap).
\end{lem}
\begin{proof}
	Say $\fR=(y_0,Y_1,y_1,Y_2,y_2,\dots,Y_n,y_n)$.
Let $x_1\in Y_1\cap Y_2$ be a vertex and $e$ an edge in $Y_1$ incident to $x_1$ whose interior lies outside of $Y_2$, and suppose that the hyperplane $H=H(e)$ intersects $Y_i$ for some $2\leq i\leq n-1$ (if no such $x_1$ and $e$ exist then $\fR$ already satisfies property (Trap) and we are done).
Let $e_i$ be an edge in $Y_i$ which is dual to $H$, and let $x_i$ be a vertex incident to $e_i$.
In particular, $x_i\in N(H)\cap Y_i$ (recall that $N(H)$ is the carrier of $H$, see Definition \ref{defn:carrier}).
Consider the route
$$\fR(x_1,e,i,x_i)=(x_1,Y_2,y_2,Y_3,y_3,\dots,Y_i,x_i,N(H),x_1).$$
$\fR(x_1,e,i,x_i)$ is clearly a closed embedded route of length $i$. 

We claim that $\fR(x_1,e,i,x_i)$ is essential.
Indeed, if not then Lemma \ref{lem:routeelev}\ref{item:essential} implies that there is a closed elevation
$$\wt{\fR}(x_1,e,i,x_i)=(\tilde{x}_1,\wt{Y}_2,\tilde{y}_2,\wt{Y}_3,\tilde{y}_3,\dots,\wt{Y}_i,\tilde{x}_i,N(\wt{H}),\tilde{x}_1)$$
of $\fR(x_1,e,i,x_i)$ to the universal cover $\wt{X}\to X$. 
Note that $\wt{H}$ is a hyperplane of $\wt{X}$ (any elevation of $N(H)$ to $\wt{X}$ is a hyperplane carrier in $\wt{X}$ by Remark \ref{remk:carrierelev}).
Note also that the osculation of $H$ and $Y_2$ at $(x_1,e)$ lifts to an osculation of $\wt{H}$ and $\wt{Y}_2$; since convex subcomplexes of CAT(0) cube complexes do not inter-osculate with hyperplanes, we see that $\wt{H}$ does not intersect $\wt{Y}_2$.
In addition, $e_i$ lifts to an edge in $\wt{Y}_i$ which is dual to $\wt{H}$ and incident to $\ti{x}_i$, so $\wt{H}$ intersects $\wt{Y}_i$.
If $i=2$ then we get a contradiction straight away, otherwise use Lemma \ref{lem:routeelev}\ref{item:extendelev} to extend $(\tilde{x}_1,\wt{Y}_2,\tilde{y}_2,\wt{Y}_3,\tilde{y}_3,\dots,\wt{Y}_i)$ to an elevation of $\fR$ to $\wt{X}$. We know that $\wt{H}\in\hH(\wt{Y}_1)\cap\hH(\wt{Y}_i) - \hH(\wt{Y}_2)$, so $\fR$ fails property (Hyp$_i$), a contradiction. This proves the claim.

We have assumed that Theorem \ref{thm:embroutes} holds for routes of length less than $n$, therefore it holds for $\fR(x_1,e,i,x_i)$, and there exists a finite cover $\hat{X}\to X$ with no closed elevations of $\fR(x_1,e,i,x_i)$.
Moreover, we may choose $\hat{X}$ so that this works for all choices of $x_1$, $e$, $i$ and $x_i$ as above (have one finite cover for each choice of $x_1,e,i,x_i$, then take a finite cover that factors through all of them).
	
Finally, we claim that each closed elevation $\hat{\fR}=(\hat{y}_0,\hat{Y}_1,\hat{y}_1,\hat{Y}_2,\hat{y}_2,\dots,\hat{Y}_n,\hat{y}_n)$ of $\fR$ to $\hat{X}$ satisfies property (Trap) (it will be embedded since we assumed $\fR$ is embedded).
Suppose for contradiction that $\hat{x}_1\in\hat{Y}_1\cap\hat{Y}_2$ is a vertex and $\hat{e}$ is an edge in $\hat{Y}_1$ incident to $\hat{x}_1$ whose interior lies outside of $\hat{Y}_2$, and so that the hyperplane $\hat{H}=H(\hat{e})$ intersects $\hat{Y}_i$ for some $2\leq i\leq n-1$.
Let $\hat{e}_i$ be an edge in $\hat{Y}_i$ which is dual to $\hat{H}$, and let $\hat{x}_i$ be a vertex incident to $\hat{e}_i$.
Say $\hat{x}_1$, $\hat{e}$ and $\hat{x}_i$ descend to $x_1$, $e$ and $x_i$ in $X$.
But then $(\hat{x}_1,\hat{Y}_2,\hat{y}_2,\hat{Y}_3,\hat{y}_3,\dots,\hat{Y}_i,\hat{x}_i,N(\hat{H}),\hat{x}_1)$ is a closed elevation of the route $\fR(x_1,e,i,x_i)$ to $\hat{X}$, contrary to the construction of $\hat{X}$.
\end{proof}

\bigskip
\section{The Walker and Imitator construction}\label{sec:walker}

In this section we complete the proof of Theorem \ref{thm:embroutes}.
The main tool we will use in this section is the Walker and Imitator construction from \cite{Shepherd23}.
The construction is essentially equivalent to the canonical completion and retraction of Haglund--Wise \cite{HaglundWise08}, but the interpretation in terms of walker and imitator will better facilitate our arguments. 
Our description of the walker-imitator construction here will be fairly brief, we refer the reader to \cite{Shepherd23} for further discussion and examples.
For simplicity we will describe the construction for the special case of a subcomplex (which is the only case we need) rather than for a general local isometry of cube complexes.

\begin{cons}(Walker and Imitator)\\\label{cons:imitator}
	Let $Y$ be a locally convex subcomplex of a directly special cube complex $X$. We consider two people wandering around the 1-skeleta of $X$ and $Y$: the \emph{walker} wanders around $X$ while the \emph{imitator} wanders around $Y$. The walker wanders freely around the 1-skeleton of $X$, while the imitator tries to ``copy'' the walker in the following way: if the imitator and walker are at vertices $(y,x)\in Y\times X$ and the walker traverses an edge $e$ incident to $x$, then the imitator traverses the edge $f$ incident to $y$ with $f\parallel e$, if such an edge exists, otherwise they remain at $y$ (see Definition \ref{defn:parallelism} for the definition of parallelism). Note that the edge $f$, if it exists, will be unique because $H(e)$ doesn't self-osculate at $y$. 
	
	Iterating the process, if the walker travels along a path $\gamma$, starting at $x$, then the imitator travels along a path that we denote $\delta=\delta(\gamma,y)$, starting at $y$. 
\end{cons}

\begin{lem}\label{lem:gammasquare}\cite[Lemma 3.2]{Shepherd23}\\
	If $\gamma$ and $\gamma'$ are homotopic paths in $X$ (rel endpoints), then for any vertex $y\in Y$ the paths $\delta(\gamma,y)$ and $\delta(\gamma',y)$ are homotopic in $Y$.
\end{lem}

\begin{cons}(Imitator covers)\\\label{cons:cover}
	Let $X$ and $Y$ be as in Construction \ref{cons:imitator} and pick a base vertex $y\in Y$. We have a right-action of $G:=\pi_1(X,y)$ on the vertex set of $Y$, defined as follows: for $y'\in Y$ a vertex and $[\gamma]\in G$, let $y'\cdot[\gamma]$ be the endpoint of the path $\delta(\gamma,y')$ -- this only depends on the homotopy class of $\gamma$ by Lemma \ref{lem:gammasquare}. Let $G_y<G$ be the stabiliser of $y$, which has finite index in $G$ if $Y$ is finite. 
	Define the \emph{imitator cover of $(X,Y,y)$} as the based cover $(\dot{X},\dot{y})\to(X,y)$ that corresponds to the subgroup $G_y$.
\end{cons}

\begin{remk}\label{remk:notclosedlift}
	If $\gamma$ is a loop in $X$ based at $y$ such that $\delta(\gamma,y)$ does not end at $y$, then $[\gamma]\notin G_y$, so the lift of $\gamma$ to the imitator cover $\dot{X}$ based at $\dot{y}$ does not close up as a loop.
\end{remk}

The following lemma enables us to control the movement of the imitator in certain situations. This lemma is a slight variation on \cite[Lemma 4.5]{Shepherd23}.

\begin{lem}(Subcomplex Entrapment)\\\label{lem:subcomplextrap}
	Let $Y$ be a locally convex subcomplex of a directly special cube complex $X$.
	Let the walker be in $X$ and the imitator in $Y$.
	Let $Z\subset X$ be a locally convex subcomplex that does not inter-osculate with hyperplanes of $X$. Suppose the imitator and walker start at positions $(y,x)\in Y\times X$.
	\begin{enumerate}
		\item\label{item:stayin} If $y,x\in Z$, then the imitator stays inside $Y\cap Z$ as long as the walker stays inside $Z$.
		\item\label{item:stayout} If $x\in Z$ but $y\notin Z$, then the imitator stays outside of $Z$ as long as the walker stays inside $Z$.
	\end{enumerate}
\end{lem}
\begin{proof}
The only way that the imitator can switch between the inside and outside of $Z$ is by traversing an edge $f$ in $Y$ with exactly one endpoint in $Z$.
And this can only happen if the walker traverses an edge $e$ which is parallel to $f$.
Since $Z$ does not inter-osculate with the hyperplane $H(f)$, there is no edge in $Z$ which is parallel to $f$; so, as long as the walker stays inside $Z$, it is impossible for the imitator to  switch between the inside and outside of $Z$.
\end{proof}

We are now ready to apply the Walker and Imitator construction to prove a lemma about routes in directly special cube complexes.

\begin{lem}\label{lem:routetrap}
Let $\fR=(y_0,Y_1,y_1,Y_2,y_2,\dots,Y_n,y_n)$ be an embedded closed route in a finite directly special cube complex $X$ with $n\geq4$.
Suppose that $\fR$ satisfies property (Trap), $Y_n$ does not inter-osculate with hyperplanes and $Y_2\cap Y_n=\emptyset$.
Then there is a finite cover $\hat{X}\to X$ with no closed elevations of $\fR$.
\end{lem}
\begin{proof}
We work with the walker in $X$ and the imitator in $Y_1$.
Let $(\dot{X},\dot{y}_0)\to(X,y_0)$ be the imitator cover of $(X,Y_1,y_0)$.
Let $\gamma=(\gamma_1,\gamma_2,\dots,\gamma_n)$ be a path along $\fR$.

We claim that $\delta(\gamma_1\gamma_2\cdots\gamma_n,y_0)$ does not end at $y_0$, i.e. if the walker traverses $\gamma_1\gamma_2\cdots\gamma_n$ then the imitator does not end up back at $y_0$.
Indeed, as the walker traverses the first segment $\gamma_1$, the imitator also traverses $\gamma_1$ (since $\gamma_1$ is in $Y_1$).
Recall that property (Trap) means that, for any vertex $y\in Y_1\cap Y_2$ and edge $e$ in $Y_1$ incident to $y$ whose interior lies outside $Y_2$, the hyperplane $H=H(e)$ is disjoint from $Y_2,Y_3,\dots,Y_{n-1}$.
As the walker traverses $\gamma_2\gamma_3\cdots\gamma_{n-1}$, the only way the imitator can leave $Y_2$ is by traversing an edge $e$ as above, and this can only happen if the walker traverses an edge parallel to $e$.
So it follows from property (Trap) that the imitator stays trapped in $Y_1\cap Y_2$ as the walker traverses $\gamma_2\gamma_3\cdots\gamma_{n-1}$.
Since $Y_2\cap Y_n=\emptyset$, the imitator is outside $Y_n$ when the walker reaches the end of $\gamma_2\gamma_3\cdots\gamma_{n-1}$.
As the walker traverses the final segment $\gamma_n$ (which is in $Y_n$), the imitator must stay outside of $Y_n$ because of Subcomplex Entrapment and the fact that $Y_n$ does not inter-osculate with hyperplanes.
As $y_0=y_n\in Y_n$, we conclude that the imitator does not end up back at $y_0$ when the walker reaches the end of $\gamma_n$. This proves the claim.

By Remark \ref{remk:notclosedlift}, the lift of $\gamma_1\gamma_2\cdots\gamma_n$ to $\dot{X}$ based at $\dot{y}_0$ does not close up as a loop.
Now let $\hat{X}\to X$ be a finite regular cover that factors through $\dot{X}$.
Observe that $\gamma_1\gamma_2\cdots\gamma_n$ (and hence $\gamma$) has no closed lift to $\hat{X}$.
This works for any choice of $\gamma$, so no path along $\fR$ has a closed lift to $\hat{X}$.
Finally, it follows from Lemma \ref{lem:routeelev}\ref{item:elevtopath} that $\hat{X}$ has no closed elevations of $\fR$.
\end{proof}

We can now prove Theorem \ref{thm:embroutes}.

\theoremstyle{plain}
\newtheorem*{thm:embroutes}{Theorem \ref{thm:embroutes}}
\begin{thm:embroutes}
	Let $X$ be a finite directly special cube complex and let $n\geq1$ be an integer.
	For any essential embedded closed route $\fR$ in $X$ of length $n$, there exists a finite cover $\hat{X}\to X$ with no closed elevations of $\fR$. 
\end{thm:embroutes}
\begin{proof}
We induct on $n$.
The case $n\leq3$ follows from the separability of products of 3 convex-cocompact subgroups in $\pi_1(X)$ \cite[Proposition 4.13]{Shepherd23} together with Proposition \ref{prop:routesep}.
Now let $n\geq4$, and assume that the theorem holds for routes of length less than $n$.
Let $\fR$ be an essential embedded closed route in $X$ of length $n$.
By Lemma \ref{lem:inductHyp}, we may assume that $\fR$ satisfies property (Hyp$_j$) for all $2\leq j\leq n-1$.

It suffices to find a finite cover $\hat{X}\to X$ such that each closed elevation $$\hat{\fR}=(\hat{y}_0,\hat{Y}_1,\hat{y}_1,\hat{Y}_2,\hat{y}_2,\dots,\hat{Y}_n,\hat{y}_n)$$ of $\fR$ to $\hat{X}$ satisfies the three assumptions of Lemma \ref{lem:routetrap}, as then we can apply Lemma \ref{lem:routetrap} to each of these elevations $\hat{\fR}$ to obtain a further finite cover of $X$ with no closed elevations of $\fR$ (using Lemma \ref{lem:factorarg}). 
By Lemma \ref{lem:propTrap}, there exists a finite cover $\hat{X}_1\to X$ such that each closed elevation of $\fR$ satisfies property (Trap).
Proposition \ref{prop:routeembed} gives us a finite cover $\hat{X}_2\to X$ such that each closed elevation $\hat{\fR}$ of $\fR$ has $\hat{Y}_n$ not inter-osculating with hyperplanes.
And Lemma \ref{lem:2intn} provides a finite cover $\hat{X}_3\to X$ such that each closed elevation $\hat{\fR}$ of $\fR$ satisfies $\hat{Y}_2\cap\hat{Y}_n=\emptyset$.
Finally, choosing $\hat{X}$ to be a finite cover of $X$ that factors through $\hat{X}_1$, $\hat{X}_2$ and $\hat{X}_3$, we ensure that all closed elevations of $\fR$ to $\hat{X}$ satisfy all three of the assumptions of Lemma \ref{lem:routetrap}, as required (it is straightforward to check that these three properties of routes pass to elevations).
\end{proof}

\bigskip
\section{Proof of Theorem \ref{thm:prodsep}}\label{sec:proof}

The goal of this section is to prove Theorems \ref{thm:prodsep} and \ref{thm:routesext}.
First, we weaken the assumptions of Theorem \ref{thm:embroutes} so that $X$ is virtually special instead of directly special, and so that the route $\fR$ is finite instead of embedded.

\begin{thm}\label{thm:routes}
	Let $X$ be a finite virtually special cube complex.
	For any finite essential closed route $\fR$ in $X$, there exists a finite cover $\hat{X}\to X$ with no closed elevations of $\fR$. 
\end{thm}
\begin{proof}
	Let $\fR$ be a finite essential closed route in $X$.	
	By Proposition \ref{prop:routeembed}, there is a finite cover $X'\to X$ where every elevation of $\fR$ is embedded.
	The closed elevations of $\fR$ to $X'$ will also be essential by Lemma \ref{lem:routeelev}\ref{item:esselev}.
	Passing to a further finite cover if necessary, we may assume that $X'$ is directly special.
	Now apply Theorem \ref{thm:embroutes} to each closed elevation of $\fR$ to $X'$.
	By Lemma \ref{lem:factorarg}, we produce a finite cover $\hat{X}\to X$ with no closed elevations of $\fR$.
\end{proof}

We can now prove Theorem \ref{thm:routesext}, which further weakens the assumptions of Theorem \ref{thm:routes} by removing the assumption that $X$ is finite.

\begin{thm}\label{thm:routesext}
	Let $X$ be a virtually special cube complex.
	For any finite essential closed route $\fR$ in $X$, there exists a finite-sheeted cover $\hat{X}\to X$ with no closed elevations of $\fR$. 
\end{thm}
\begin{proof}
By Lemma \ref{lem:factorarg}, we may assume that $X$ is special.
Thus there exists a local isometry $X\to X_\Gamma$ to the Salvetti complex of a right-angled Artin group.
Let $\fR=(y_0,Y_1,y_1,Y_2,y_2,\dots,Y_n,y_n)$ be a finite essential closed route in $X$.
Each $Y_i$ has compact image in $X_\Gamma$, hence there is a finite subgraph $\La\subset\Gamma$ such that the images of the $Y_i$ are all contained in the natural subcomplex $X_\La\subset X_\Gamma$.
There is also a natural retraction $X_\Gamma\to X_\La$ by collapsing all edges labelled by vertices in $\Gamma-\La$.
This gives us a map $X\to X_\La$ such that each composition $Y_i\to X\to X_\La$ is a local isometry.
As a result, the route $\fR$ induces a finite closed route $\fR'$ in $X_\La$.

We claim that $\fR'$ is essential.
Indeed, any path $\delta'$ along $\fR'$ lifts to a path $\delta$ along $\fR$ (lift each segment of $\delta'$ to the corresponding $Y_i$, then project it to $X$).
The route $\fR$ is essential, so $\delta$ is essential. Since $X\to X_\Gamma$ is a local isometry, the image of $\delta$ in $X_\Gamma$ is essential, and this is the same as the image of $\delta'$ under the inclusion map $X_\La\xhookrightarrow{} X_\Gamma$.
It follows that $\delta'$ is essential, as required.

Now apply Theorem \ref{thm:routes} to $\fR'$ and $X_\La$, producing a finite cover $\hat{X}_\La\to X_\La$ with no closed elevations of $\fR'$.
Let $\hat{X}$ be a component of the pullback of the maps $X\to X_\La$ and $\hat{X}_\La\to X_\La$.
This is similar to the construction of an elevation, except that $X\to X_\La$ is not necessarily a local isometry.
As a result, the map $\hat{X}\to\hat{X}_\La$ might not be a local isometry either, but the map $\hat{X}\to X$ will be a finite-sheeted covering (since $\hat{X}_\La\to X_\La$ is a finite-sheeted covering).

Now let $\hat{Y}_i\to\hat{X}$ be an elevation of $Y_i\to X$. This yields the following commutative diagram:

\begin{equation}\label{doublepullback}
	\begin{tikzcd}[
		ar symbol/.style = {draw=none,"#1" description,sloped},
		isomorphic/.style = {ar symbol={\cong}},
		equals/.style = {ar symbol={=}},
		subset/.style = {ar symbol={\subset}}
		]
		\hat{Y}_i\ar{d}\ar{r}&\hat{X}\ar{d}\ar{r}&\hat{X}_\La\ar{d}\\
		Y_i\ar{r}&X\ar{r}&X_\La
	\end{tikzcd}
\end{equation}
The left and right hand squares in (\ref{doublepullback}) are both pullback diagrams (restricted to one component), so they compose to give a third pullback diagram, i.e. $\hat{Y}_i$ is a component of the pullback of $Y_i\to X_\La$ and $\hat{X}_\La\to X_\La$.
As $Y_i\to X_\La$ is a local isometry, the map $\hat{Y}_i\to \hat{X}_\La$ is an elevation of $Y_i\to X_\La$.
It follows that any elevation $\hat{\fR}$ of $\fR$ to $\hat{X}$ induces an elevation $\hat{\fR}'$ of $\fR'$ to $\hat{X}_\La$.
Moreover, if $\hat{\fR}$ is closed, then its initial and terminal vertices are the same vertex in $\hat{X}$, so they map to the same vertex in $\hat{X}_\La$, hence $\hat{\fR}'$ is also closed.
By construction, $\hat{X}_\La$ has no closed elevations of $\fR'$, therefore $\hat{X}$ has no closed elevations of $\fR$.
\end{proof}

It is now easy to deduce Theorem \ref{thm:prodsep}.

\theoremstyle{plain}
\newtheorem*{thm:prodsep}{Theorem \ref{thm:prodsep}}
\begin{thm:prodsep}
	Let $G\acts \wt{X}$ be a virtually special action of a group $G$ on a CAT(0) cube complex $\wt{X}$, and let $K_1,K_2,\dots,K_n<G$ be convex-cocompact subgroups. Then the product $K_1K_2\cdots K_n$ is separable in $G$.
\end{thm:prodsep}
\begin{proof}
As the action is virtually special, there is a finite-index subgroup $G'<G$ acting freely on $\wt{X}$, with the quotient $\wt{X}/G'$ a special cube complex. Passing to the normal core, we may assume that $G'$ is normal in $G$.
By Lemma \ref{lem:sepfiniteindex}, we may assume that $G$ is equal to $G'$.
The result then follows from Proposition \ref{prop:routesep} and Theorem \ref{thm:routesext}.
\end{proof}

\bigskip
\section{Separability with respect to group actions}\label{sec:actions}

In this section we prove Proposition \ref{prop:ballsep} and Corollary \ref{cor:contact} regarding separable subsets in group actions.

Recall that, given an action of a group $G$ on a set $A$, we say that a subset $B\subset A$ is \emph{$G$-separable} if for any $a\in A-B$ there exists a finite-index subgroup $\hat{G}<G$ such that $\hat{G}\cdot a \cap B=\emptyset$.
In the introduction we already discussed the connection between separability of subsets of $G$ and $G$-separability of subsets of $A$ in some special cases.
The following lemma illuminates this connection in a more general setting.

\begin{lem}\label{lem:Bsep}
	Suppose $G$ acts on a set $A$ with finitely many orbits, and let $B\subset A$.
	Then $B$ is $G$-separable if and only if it can be written as a finite union $B=\cup_i S_i\cdot b_i$ with $b_i\in B$, $S_i\subset G$ separable, and $S_i G_{b_i}=S_i$.
\end{lem}
\begin{proof}
$\implies$: First suppose that $B$ is $G$-separable. Let $\{b_i\}\subset B$ be a collection of representatives for the $G$-orbits of $A$ that intersect $B$.  We assumed that $A$ has finitely many $G$-orbits, so the collection $\{b_i\}$ is finite. For each $b_i$, define
$$S_i=\{g\in G\mid gb_i\in B\}.$$
Clearly we can write $B$ as a finite union $B=\cup_i S_i\cdot b_i$, and $S_i G_{b_i}=S_i$ for each $i$.
It remains to show that each $S_i$ is separable in $G$.
Let $g\in G-S_i$.
By definition, $gb_i\notin B$.
We assumed that $B$ is $G$-separable, so there exists a finite-index subgroup $\hat{G}<G$ such that $\hat{G}\cdot gb_i\cap B=\emptyset$.
It follows that $\hat{G}g\cap S_i=\emptyset$, thus $S_i$ is separable (using the version of separability from Lemma \ref{lem:sepequiv}\ref{it:sep3}).

$\impliedby$: Now suppose that $B$ can be written as a finite union $B=\cup_i S_i\cdot b_i$ with $b_i\in B$, $S_i\subset G$ separable, and $S_i G_{b_i}=S_i$.
Let $a\in A-B$, and pick one of the elements $b_i$.

First suppose that $a$ is in the same $G$-orbit as $b_i$, so we may write $a=g_ib_i$ for some $g_i\in G$.
Then $g_i\notin S_i$, so, by separability of $S_i$, there exists a finite-index subgroup $\hat{G}_i<G$ such that $\hat{G}_ig_i\cap S_i=\emptyset$.
Since $S_iG_{b_i}=S_i$, it follows that $\hat{G}_i\cdot a\cap S_i\cdot b_i=\emptyset$.
If on the other hand $a$ is in a different $G$-orbit to $b_i$, then we can simply put $\hat{G}_i=G$ and we still have $\hat{G}_i\cdot a\cap S_i\cdot b_i=\emptyset$.

Finally, define $\hat{G}=\cap_i\hat{G}_i$, which is again a finite-index subgroup of $G$.
As $B=\cup_i S_i\cdot b_i$ and $\hat{G}_i\cdot a\cap S_i\cdot b_i=\emptyset$ for each $i$, we have $\hat{G}\cdot a\cap B=\emptyset$.
Thus $B$ is $G$-separable.
\end{proof}

We now use this lemma to prove Proposition \ref{prop:ballsep}, which explains how the separability of certain products of subgroups in $G$ relates to the $G$-separability of balls in graphs that admit a cocompact $G$-action.

\theoremstyle{plain}
\newtheorem*{prop:ballsep}{Proposition \ref{prop:ballsep}}
\begin{prop:ballsep}
	Let $G$ act cocompactly on a graph $X$. Suppose that any product of vertex stabilisers is separable in $G$.
	Then for any $x\in VX$ and any integer $R\geq0$, the ball
	$$B_R(x)=\{y\in VX\mid d(x,y)\leq R\}$$
	is $G$-separable.
\end{prop:ballsep}
\begin{proof}
	For each $y\in B_{R-1}(x)$, let $N_y$ be a set of $G_y$-orbit representatives for the neighbours of $y$.
	Since $G$ acts cocompactly on $X$, we see that each $N_y$ is finite.
	Now let
	$$\Omega=\{(x=x_0,x_1,x_2,\dots,x_r)\mid 0\leq r\leq R,\, x_{i+1}\in N_{x_i}\}.$$
	Note that $\Omega$ is also finite.
	We claim that
	\begin{equation}\label{BRx}
		B_R(x)=\bigcup_{(x_0,\dots,x_r)\in\Omega} G_{x_0}G_{x_1}\cdots G_{x_r}\cdot x_r.
	\end{equation}
The inclusion $\supset$ is clear.
We now prove $\subset$, so let $y\in B_R(x)$.
Let $(x=y_0,y_1,y_2,\dots,y_r=y)$ be the vertices on a shortest path from $x$ to $y$ in $X$ (so $0\leq r\leq R$).
Inductively define vertices $x=x_0,x_1,x_2,\dots,x_r$ and elements $g_i\in G_{x_{i}}$ such that $g_ig_{i-1}\cdots g_0y_{i+1}=x_{i+1}\in N_{x_i}$ ($0\leq i\leq r-1$).
Indeed, for the base case we just set $x_0=x$.
For the inductive step, suppose we have already chosen $x_0,\dots,x_i$ and $g_0,\dots,g_{i-1}$ with the required properties ($0\leq i\leq r-1$).
So $g_{i-1}g_{i-2}\cdots g_0y_i=x_i\in B_{R-1}(x)$ is a neighbour of $g_{i-1}g_{i-2}\cdots g_0 y_{i+1}$.
Then by definition of $N_{x_i}$, there is $g_i\in G_{x_i}$ such that $g_ig_{i-1}\cdots g_0y_{i+1}\in N_{x_i}$, and we define $x_{i+1}=g_ig_{i-1}\cdots g_0y_{i+1}$.
This completes the inductive step.
By construction, $(x_0,x_1,\dots,x_r)\in\Omega$ and $x_r=g_{r-1}g_{r-2}\cdots g_0y_r$, hence
$$y=y_r=g_0g_1\cdots g_{r-1}x_r\in G_{x_0}G_{x_1}\cdots G_{x_r}\cdot x_r$$
is an element of the right-hand side of (\ref{BRx}).
This completes the proof of (\ref{BRx}). 

Finally, we apply Lemma \ref{lem:Bsep} with $A=VX$, $B=B_R(x)$ and the finite union $B=\cup_i S_i\cdot b_i$ being the union in (\ref{BRx}), noting that each $S_i$ is a product of vertex stabilisers so is separable in $G$ by assumption of the proposition.
\end{proof}

Combining Proposition \ref{prop:ballsep} with Theorem \ref{thm:prodsep} allows us to prove separability results for a wide range of actions of virtually special cubulated groups on graphs.
The following corollary provides some notable examples.

\theoremstyle{plain}
\newtheorem*{cor:contact}{Corollary \ref{cor:contact}}
\begin{cor:contact}
	Let $G\acts X$ be one of the following actions:
	\begin{enumerate}
		\item The action of a virtually special cubulated group on its contact graph \cite{Hagen14,BehrstockHagenSisto17}.
		\item The action of a finitely generated right-angled Artin group on its extension graph $\Gamma^e$ or clique graph $\Gamma^e_k$ \cite{KimKoberda13}.
		\item The action of a one-ended hyperbolic cubulated group on its canonical JSJ tree \cite{Bowditch98}.
		\item The action of a virtually special cubulated group on its coned-off Cayley graph with respect to a given convex-cocompact subgroup and finite generating set \cite{Farb98} (we make no assumption of relative hyperbolicity here).
	\end{enumerate}
	Then for any $x\in VX$ and any integer $R\geq0$, the ball $B_R(x)$ is $G$-separable.
\end{cor:contact}
\begin{proof}
	The result follows from Theorem \ref{thm:prodsep} and Proposition \ref{prop:ballsep} once we know that the action $G\acts X$ is cocompact with convex-compact vertex stabilisers.
	
	In part \ref{it:contact}, each vertex stabiliser for $G\acts X$ is a hyperplane stabiliser with respect to the cubulation of $G$, hence convex-cocompact. Edges in $X$ correspond to intersections or osculations of hyperplanes, so the cocompactness of $G\acts X$ follows from the cocompactness of the cubulation of $G$.
	
	In part \ref{it:extension}, first consider the action of a finitely generated right-angled Artin group $G=A_\Gamma$ on its extension graph $X=\Gamma^e$.
	The vertices of $\Gamma^e$ correspond to conjugates of the standard generators of $A_\Gamma$, and the action is by conjugation.
	For a generator $v\in V\Gamma$, the stabiliser of $v$ as a vertex of $\Gamma^e$ is equal to the centraliser of $v$ as an element of $A_\Gamma$, which is equal to the subgroup of $A_\Gamma$ generated by the standard generators that commute with $v$, which is a convex-compact subgroup (with respect to the standard cubulation of $A_\Gamma$). Hence $A_\Gamma\acts \Gamma^e$ has convex-cocompact vertex stabilisers.
	By \cite[Proposition 7.1]{Koberda13}, each edge of $\Gamma^e$ is in the same $A_\Gamma$-orbit as an edge of the form $(u,v)$, with $u,v\in V\Gamma$ commuting generators of $A_\Gamma$, therefore the action of $A_\Gamma$ on $\Gamma^e$ is cocompact.
	In a similar vein, one can show that each clique of $\Gamma^e$ is in the same $A_\Gamma$-orbit as a clique spanned by pairwise commuting standard generators of $A_\Gamma$. It follows that the action of $A_\Gamma$ on the clique graph $\Gamma^e_k$ is cocompact and has convex-cocompact vertex stabilisers.
	
	In part \ref{it:JSJ}, the vertex stabilisers of $G\acts X$ are quasiconvex, hence convex-compact with respect to any cubulation of $G$ by \cite[Proposition 7.2]{HaglundWise08}. Moreover, $G$ acts on $X$ cocompactly by construction.
	Also note that all cubulations of $G$ are virtually special because $G$ is hyperbolic \cite{Agol13}.
	
	Finally, in part \ref{it:coned}, each vertex stabiliser is either trivial or a conjugate of the given convex-compact subgroup, and the action is cocompact by construction.
\end{proof}

\bibliographystyle{alpha}
\bibliography{Ref}

\end{document}